\documentclass[12pt, a4paper, reqno]{amsart}
\usepackage{amsmath}
\usepackage{amsfonts}
\usepackage{amssymb}
\usepackage{color}
\usepackage{comment,graphicx}
\usepackage[T1]{fontenc}
\usepackage{url}
\usepackage{ae}
\usepackage{mathtools}
\usepackage{esint}
\usepackage{latexsym}
\usepackage{amscd}

\def\phi{\varphi}
\def\rho{\varrho}
\def\epsilon{\varepsilon}
\numberwithin{equation}{section}
\theoremstyle{plain}
\newtheorem{theorem}[equation]{Theorem}
\newtheorem{lemma}[equation]{Lemma}
\newtheorem{proposition}[equation]{Proposition}
\newtheorem{corollary}[equation]{Corollary}
\theoremstyle{definition}
\newtheorem{definition}[equation]{Definition}

\theoremstyle{remark}
\newtheorem{remark}[equation]{Remark}
\newtheorem{example}[equation]{Example}

\renewcommand{\leq}{\leqslant}
\renewcommand{\geq}{\geqslant}

\begin{document}
\title[Weigthed function spaces]{On the function spaces of general weights}
\author[D. Drihem]{Douadi Drihem}
\address{Douadi Drihem\\
M'sila University\\
Department of Mathematics\\
Laboratory of Functional Analysis and Geometry of Spaces\\
M'sila 28000, Algeria.}
\email{douadidr@yahoo.fr, douadi.drihem@univ-msila.dz}
\thanks{ }
\date{\today }

\begin{abstract}
The aim of this paper is twofold. Firstly, we chatacterize the Besov spaces $%
\dot{B}_{p,q}(\mathbb{R}^{n},\{t_{k}\})$ and the Triebel-Lizorkin spaces $%
\dot{F}_{p,q}(\mathbb{R}^{n},\{t_{k}\})$ for $q=\infty $. Secondly, under
some suitable assumptions on the $p$-admissible weight sequence $\{t_{k}\}$,
we prove that%
\begin{equation*}
\dot{A}_{p,q}(\mathbb{R}^{n},\{t_{k}\})=\dot{A}_{p,q}(\mathbb{R}%
^{n},t_{j}),\quad j\in \mathbb{Z},
\end{equation*}%
in the sense of equivalent quasi-norms, with $\dot{A}$ $\in \{\dot{B},\dot{F}%
\}$. Moreover, we find a necessary and sufficient conditions for the
coincidence of the spaces $\dot{A}_{p,q}(\mathbb{R}^{n},t_{i}),i\in \{1,2\}$.
\end{abstract}

\maketitle

\section{Introduction}

Function spaces of generalized smoothness have been introduced by several
authors. We refer, for instance, to Cobos and Fernandez \cite{CF88}, Goldman 
\cite{Go79} and \cite{Go83}, and Kalyabin \cite{Ka83}; see also Kalyabin and
Lizorkin \cite{Kl87}.

Besov \cite{B03}, \cite{B05.1} and \cite{B05.2} defined function spaces of
variable smoothness and obtained their characterizations by differences,
interpolation, embeddings and extension. Such spaces are a special case of
the so-called 2-microlocal function spaces. The concept of 2-microlocal
analysis, or 2-microlocal function spaces, is due to Bony \cite{Bo84}. These
type of function spaces have been studied in detail in \cite{K8}.

Dominguez and Tikhonov \cite{DT} gave a treatment of function spaces with
logarithmic smoothness (Besov, Sobolev, Triebel-Lizorkin), including various
new characterizations for Besov norms in terms of different, sharp estimates
for Besov norms of derivatives and potential operators (Riesz and Bessel
potentials) in terms of norms of functions themselves and sharp embeddings
between the Besov spaces defined by differences and by Fourier-analytical
decompositions as well as between Besov and Sobolev/Triebel-Lizorkin spaces.

More general Besov spaces with variable smoothness were explicitly studied
by Ansorena and Blasco \cite{AB96}, including atomic decomposition. The
wavelet decomposition (with respect to a compactly supported wavelet basis
of Daubechies type) of nonhomogeneous Besov spaces of generalized smoothness
was achieved by Almeida \cite{Al}.

More general function spaces of generalized smoothness can be found in \cite%
{BoHo06}, \cite{FL06} and \cite{HaS08}, and reference therein.

The theory of spaces of generalized smoothness\ had a remarkable development
in part due to its usefulness in applications. For instance, they appear in
the study of trace spaces on fractals, see Edmunds and Triebel \cite{ET96}
and \cite{ET99}, where they introduced the spaces $B_{p,q}^{s,\Psi }$, where 
$\Psi $ is a so-called admissible function, typically of log-type near $0$.
For a complete treatment of these spaces we refer the readers to the work of
Moura \cite{Mo01}.

Tyulenev has introduced in \cite{Ty15} a new family of Besov spaces of
variable smoothness which cover many classes of Besov spaces, where the norm
on these spaces was defined\ with the help of classical differences.

In \cite{D20} and \cite{D20.1} the author introduced Besov and
Triebel-Lizorkin spaces $\dot{B}_{p,q}(\mathbb{R}^{n},\{t_{k}\})$ and $\dot{F%
}_{p,q}(\mathbb{R}^{n},\{t_{k}\})$ with general smoothness and presented
some their properties, such as the $\varphi $-transform characterization in
the sense of Frazier and Jawerth, the smooth atomic, molecular and wavelet
decomposition, and the characterization of Besov spaces $\dot{B}_{p,q}(%
\mathbb{R}^{n},\{t_{k}\})$\ in terms of the difference relations. Also,
Sobolev type embeddings are given. These studies were all restricted to
bounded exponents $q$.

The purpose of the present paper is to extend the results of \cite{D20.1} to
the case $q=\infty $. In addition, we present some properties of the spaces $%
\dot{F}_{p,q}(\mathbb{R}^{n},\{t_{k}\})$. More precisely, under some
suitable assumptions on the $p$-admissible weight sequence $\{t_{k}\}$ we
prove that%
\begin{equation*}
\dot{F}_{p,q}(\mathbb{R}^{n},\{t_{k}\})=\dot{F}_{p,q}(\mathbb{R}%
^{n},t_{j}),\quad j\in \mathbb{Z},
\end{equation*}%
in the sense of equivalent quasi-norms. With the help of this result, some
remarks on weighted Hardy spaces are given. Moreover, the coincidence of the
spaces $\dot{A}_{p,q}(\mathbb{R}^{n},t_{i}),i\in \{1,2\}$ is proved, which
is probably not available in the existing literature. All results of this
paper can be extended to inhomogeneous Besov and Triebel-Lizorkin spaces of
general weights.

Throughout this paper, we denote by $\mathbb{R}^{n}$ the $n$-dimensional
real Euclidean space, $\mathbb{N}$ the collection of all natural numbers and 
$\mathbb{N}_{0}=\mathbb{N}\cup \{0\}$. The letter $\mathbb{Z}$ stands for
the set of all integer numbers.\ The expression $f\lesssim g$ means that $%
f\leq c\,g$ for some independent constant $c$ (and non-negative functions $f$
and $g$), and $f\approx g$ means $f\lesssim g\lesssim f$.

For $x\in \mathbb{R}^{n}$ and $r>0$ we denote by $B(x,r)$ the open ball in $%
\mathbb{R}^{n}$ with center $x$ and radius $r$. By supp$f$ we denote the
support of the function $f$, i.e., the closure of its non-zero set. If $%
E\subset {\mathbb{R}^{n}}$ is a measurable set, then $|E|$ stands for the
(Lebesgue) measure of $E$ and $\chi _{E}$ denotes its characteristic
function. By $c$ we denote generic positive constants, which may have
different values at different occurrences.

A weight is a nonnegative locally integrable function on $\mathbb{R}^{n}$
that takes values in $(0,\infty )$ almost everywhere. For measurable set $%
E\subset \mathbb{R}^{n}$ and a weight $\gamma $, $\gamma (E)$ denotes 
\begin{equation*}
\int_{E}\gamma (x)dx.
\end{equation*}%
Given a measurable set $E\subset \mathbb{R}^{n}$ and $0<p\leq \infty $, we
denote by $L_{p}(E)$ the space of all functions $f:E\rightarrow \mathbb{C}$
equipped with the quasi-norm 
\begin{equation*}
\big\|f|L_{p}(E)\big\|:=\Big(\int_{E}\left\vert f(x)\right\vert ^{p}dx\Big)^{%
\frac{1}{p}}<\infty ,
\end{equation*}%
with $0<p<\infty $ and%
\begin{equation*}
\big\|f|L_{\infty }(E)\big\|:=\underset{x\in E}{\text{ess-sup}}\left\vert
f(x)\right\vert <\infty .
\end{equation*}%
For $0<p<\infty $, the space $L_{p,\infty }(E)$ is defined as the set of all
measurable functions $f$\ on $E$ such that%
\begin{equation*}
\big\|f|L_{p,\infty }(E)\big\|:=\sup_{\delta >0}\delta |\{x\in
E:|f(x)|>\delta \}|^{\frac{1}{p}}<\infty .
\end{equation*}%
Let $0<p,q\leq \infty $. The space $L_{p}(\ell _{q})$ is defined to be the
set of all sequences $\{f_{k}\}$ of functions such that%
\begin{equation*}
\big\|\{f_{k}\}|L_{p}(\ell _{q})\big\|:=\Big\|\Big(\sum_{k=-\infty }^{\infty
}|f_{k}|^{q}\Big)^{1/q}|L_{p}(\mathbb{R}^{n})\Big\|<\infty
\end{equation*}%
with the usual modifications if $q=\infty $. Similarly, define $L_{p,\infty
}(\ell _{q})$ as the space of all sequences $\{f_{k}\}$ of functions such
that%
\begin{equation*}
\big\|\{f_{k}\}|L_{p,\infty }(\ell _{q})\big\|:=\Big\|\Big(\sum_{k=-\infty
}^{\infty }|f_{k}|^{q}\Big)^{1/q}|L_{p,\infty }(\mathbb{R}^{n})\Big\|<\infty
,
\end{equation*}%
with the usual modifications if $q=\infty $. For a function $f$ in $L_{1}^{%
\mathrm{loc}}$, we set 
\begin{equation*}
M_{A}(f):=\frac{1}{|A|}\int_{A}\left\vert f(x)\right\vert dx
\end{equation*}%
for any $A\subset \mathbb{R}^{n}$. Furthermore, we put%
\begin{equation*}
M_{A,p}(f):=\Big(\frac{1}{|A|}\int_{A}\left\vert f(x)\right\vert ^{p}dx\Big)%
^{\frac{1}{p}},
\end{equation*}%
with $0<p<\infty $. Further, given a measurable set $E\subset \mathbb{R}^{n}$
and a weight $\gamma $, we denote the space of all functions $f:\mathbb{R}%
^{n}\rightarrow \mathbb{C}$ with finite quasi-norm 
\begin{equation*}
\big\|f|L_{p}(\mathbb{R}^{n},\gamma )\big\|:=\big\|f\gamma |L_{p}(\mathbb{R}%
^{n})\big\|
\end{equation*}%
by $L_{p}(\mathbb{R}^{n},\gamma )$.

If $1\leq p\leq \infty $ and $\frac{1}{p}+\frac{1}{p^{\prime }}=1$, then $%
p^{\prime }$ is called the conjugate exponent of $p$.

The symbol $\mathcal{S}(\mathbb{R}^{n})$ is used to denote the set of all
Schwartz functions on $\mathbb{R}^{n}$. In what follows, $Q$ will denote an
cube in the space $\mathbb{R}^{n}$\ with sides parallel to the coordinate
axes and $l(Q)$\ will denote the side length of the cube $Q$. For $v\in 
\mathbb{Z}$ and $m\in \mathbb{Z}^{n}$, denote by $Q_{v,m}$ the dyadic cube,%
\begin{equation*}
Q_{v,m}:=2^{-v}([0,1)^{n}+m).
\end{equation*}%
For the collection of all such cubes we use $\mathcal{Q}:=\{Q_{v,m}:v\in 
\mathbb{Z},m\in \mathbb{Z}^{n}\}$.

\section{Maximal inequalities}

\subsection{Muckenhoupt weights}

The purpose of this subsection is to review some known properties of\
Muckenhoupt class.

\begin{definition}
Let $1<p<\infty $. We say that a weight $\gamma $ belongs to the Muckenhoupt
class $A_{p}(\mathbb{R}^{n})$ if there exists a constant $C>0$ such that for
every cube $Q$ the following inequality holds 
\begin{equation}
M_{Q}(\gamma )M_{Q,\frac{p^{\prime }}{p}}(\gamma ^{-1})\leq C.
\label{Ap-constant}
\end{equation}
\end{definition}

The smallest constant $C$ for which $\mathrm{\eqref{Ap-constant}}$ holds,
denoted by $A_{p}(\gamma )$. As an example, we can take%
\begin{equation*}
\gamma (x)=|x|^{\alpha },\quad \alpha \in \mathbb{R}.
\end{equation*}%
Then $\gamma \in A_{p}(\mathbb{R}^{n})$, $1<p<\infty $, if and only if $%
-n<\alpha <n(p-1)$.

For $p=1$ we rewrite the above definition in the following way.

\begin{definition}
We say that a weight $\gamma $ belongs to the Muckenhoupt class $A_{1}(%
\mathbb{R}^{n})$ if there exists a constant $C>0$ such that for every cube $%
Q $\ and for a.e.\ $y\in Q$ the following inequality holds 
\begin{equation}
M_{Q}(\gamma )\leq C\gamma (y).  \label{A1-constant}
\end{equation}
\end{definition}

The smallest constant $C$ for which $\mathrm{\eqref{A1-constant}}$ holds,
denoted by $A_{1}(\gamma )$. The above classes have been first studied by
Muckenhoupt\ \cite{Mu72} and use to characterize the boundedness of the
Hardy-Littlewood maximal function on $L^{p}(\gamma )$, see the monographs 
\cite{GR85} and \cite{L. Graf14}\ for a complete account on the theory of
Muckenhoupt weights.

\begin{lemma}
\label{Ap-Property}Let $1\leq p<\infty $.\newline
$\mathrm{(i)}$ If $\gamma \in A_{p}(\mathbb{R}^{n})$, then for any $1\leq
p<q $, $\gamma \in A_{q}(\mathbb{R}^{n})$.\newline
$\mathrm{(ii)}$ If $\gamma \in A_{p}(\mathbb{R}^{n})$, then for any $%
0<\varepsilon \leq 1$, $\gamma ^{\varepsilon }\in A_{p}(\mathbb{R}^{n})$.%
\newline
$\mathrm{(iii)}$ Suppose that $\gamma \in A_{p}(\mathbb{R}^{n})$ for some $%
1<p<\infty $. Then there exist a $1<p_{1}<p<\infty $ such that $\gamma \in
A_{p_{1}}(\mathbb{R}^{n})$.
\end{lemma}

The following theorem gives a useful property of $A_{p}(\mathbb{R}^{n})$
weights (reverse H\"{o}lder inequality), see \cite[Chapter 7]{L. Graf14}.

\begin{theorem}
\label{reverse Holder inequality}Let $1\leq p<\infty \ $and $\gamma \in
A_{p}(\mathbb{R}^{n})$. Then there exist a constants $C>0$ and $\varepsilon
_{\gamma }>0$\ depending only on $p$ and the $A_{p}(\mathbb{R}^{n})$
constant of $\gamma $, such that for every cube $Q$, 
\begin{equation*}
M_{Q,1+\varepsilon _{\gamma }}(\gamma )\leq CM_{Q}(\gamma ).
\end{equation*}%
Moreover $\gamma ^{1+\varepsilon _{\gamma }}\in A_{p}(\mathbb{R}^{n}).$
\end{theorem}

\subsection{The weight class $\dot{X}_{\protect\alpha ,\protect\sigma ,p}$}

Let $0<p\leq \infty $. A weight sequence $\{t_{k}\}$ is called $p$%
-admissible if $t_{k}\in L_{p}^{\mathrm{loc}}(\mathbb{R}^{n})$ for all $k\in 
\mathbb{Z}$. We mention here that 
\begin{equation*}
\int_{E}t_{k}^{p}(x)dx<c(k)
\end{equation*}%
for any $k\in \mathbb{Z}$ and any compact set $E\subset \mathbb{R}^{n}$. For
a $p$-admissible weight sequence $\{t_{k}\}$\ we set%
\begin{equation*}
t_{k,m}:=\big\|t_{k}|L_{p}(Q_{k,m})\big\|,\quad k\in \mathbb{Z},m\in \mathbb{%
Z}^{n}.
\end{equation*}

Tyulenev\ \cite{Ty15}, \cite{Ty-151} and \cite{Ty-N-L} introduced the
following new weighted class\ and use it to study Besov spaces of variable
smoothness.

\begin{definition}
\label{Tyulenev-class}Let $\alpha _{1}$, $\alpha _{2}\in \mathbb{R}$, $%
p,\sigma _{1}$, $\sigma _{2}$ $\in (0,+\infty ]$, $\alpha =(\alpha
_{1},\alpha _{2})$ and let $\sigma =(\sigma _{1},\sigma _{2})$. We let $\dot{%
X}_{\alpha ,\sigma ,p}=\dot{X}_{\alpha ,\sigma ,p}(\mathbb{R}^{n})$ denote
the set of $p$-admissible weight sequences $\{t_{k}\}$ satisfying the
following conditions. There exist numbers $C_{1},C_{2}>0$ such that for any $%
k\leq j$\ and every cube $Q,$%
\begin{eqnarray}
M_{Q,p}(t_{k})M_{Q,\sigma _{1}}(t_{j}^{-1}) &\leq &C_{1}2^{\alpha _{1}(k-j)},
\label{Asum1} \\
(M_{Q,p}(t_{k}))^{-1}M_{Q,\sigma _{2}}(t_{j}) &\leq &C_{2}2^{\alpha
_{2}(j-k)}.  \label{Asum2}
\end{eqnarray}
\end{definition}

The constants $C_{1},C_{2}>0$ are independent of both the indexes $k$ and $j$%
.

\begin{remark}
$\mathrm{(i)}$\ We would like to mention that if $\{t_{k}\}$ satisfying $%
\mathrm{\eqref{Asum1}}$ with $\sigma _{1}=r\left( \frac{p}{r}\right)
^{\prime }$ and $0<r<p\leq \infty $, then $t_{k}^{p}\in A_{\frac{p}{r}}(%
\mathbb{R}^{n})$ for any $k\in \mathbb{Z}$ with\ $0<r<p<\infty $ and $%
t_{k}^{-r}\in A_{1}(\mathbb{R}^{n})$ for any $k\in \mathbb{Z}$\ with\ $%
p=\infty $.\newline
$\mathrm{(ii)}$ We say that $t_{k}\in A_{p}(\mathbb{R}^{n})$,\ $k\in \mathbb{%
Z}$, $1<p<\infty $ have the same Muckenhoupt constant if%
\begin{equation*}
A_{p}(t_{k})=c,\quad k\in \mathbb{Z},
\end{equation*}%
where $c$ is independent of $k$.\newline
$\mathrm{(iii)}$ Definition \ref{Tyulenev-class} is different from the one
used in \cite[Definition 2.7]{Ty15}, because we used the boundedness of the
maximal function on weighted Lebesgue spaces.
\end{remark}

\begin{example}
\label{Example1}Let $0<r<p<\infty $, a weight $\omega ^{p}\in A_{\frac{p}{r}%
}(\mathbb{R}^{n})$ and $\{s_{k}\}=\{2^{ks}\omega ^{p}\}_{k\in \mathbb{Z}}$, $%
s\in \mathbb{R}$. Clearly, $\{s_{k}\}_{k\in \mathbb{Z}}$ lies in $\dot{X}%
_{\alpha ,\sigma ,p}$ for $\alpha _{1}=\alpha _{2}=s$, $\sigma =(r(\frac{p}{r%
})^{\prime },p)$.
\end{example}

\begin{remark}
\label{Tyulenev-class-properties}Let $0<\theta \leq p\leq \infty $. Let $%
\alpha _{1}$, $\alpha _{2}\in \mathbb{R}$, $\sigma _{1},\sigma _{2}\in
(0,+\infty ]$, $\sigma _{2}\geq p$, $\alpha =(\alpha _{1},\alpha _{2})$ and
let $\sigma =(\sigma _{1}=\theta \left( \frac{p}{\theta }\right) ^{\prime
},\sigma _{2})$. Let a $p$-admissible weight sequence $\{t_{k}\}\in \dot{X}%
_{\alpha ,\sigma ,p}$. Then $\alpha _{2}\geq \alpha _{1}$, see \cite{D20}.
\end{remark}

As usual, we put%
\begin{equation*}
\mathcal{M(}f)(x):=\sup_{Q}\frac{1}{|Q|}\int_{Q}\left\vert f(y)\right\vert
dy,\quad f\in L_{1}^{\mathrm{loc}},
\end{equation*}%
where the supremum\ is taken over all cubes with sides parallel to the axis
and $x\in Q$. Also we set 
\begin{equation*}
\mathcal{M}_{\sigma }(f):=\left( \mathcal{M(}\left\vert f\right\vert
^{\sigma })\right) ^{\frac{1}{\sigma }},\quad 0<\sigma <\infty .
\end{equation*}

\begin{lemma}
\label{key-estimate1.2}Let $1<\theta \leq p<\infty $. Let $\{t_{k}\}$\ be a $%
p$-admissible\ weight\ sequence\ such that $t_{k}^{p}\in A_{\frac{p}{\theta }%
}(\mathbb{R}^{n})$, $k\in \mathbb{Z}$. Assume that $t_{k}^{p}$,\ $k\in 
\mathbb{Z}$, have the same Muckenhoupt constant, $A_{\frac{p}{\theta }%
}(t_{k}^{p})=C,k\in \mathbb{Z}$. Then%
\begin{equation}
\big\|\mathcal{M}(f_{k})|L_{p}(\mathbb{R}^{n},t_{k})\big\|\leq c\big\|%
f_{k}|L_{p}(\mathbb{R}^{n},t_{k})\big\|  \label{key-est1}
\end{equation}%
for all sequences of functions\ $f_{k}\in L_{p}(\mathbb{R}^{n},t_{k})$, $%
k\in \mathbb{Z}$, where $c>0$ is independent of $k$.
\end{lemma}

For the proof, see\ e.g., \cite{D20}.

\begin{remark}
\label{r-estimates copy(2)}$\mathrm{(i)}$ We would like to mention that the
result of Lemma \ref{key-estimate1.2} is true if we assume that $t_{k}\in A_{%
\frac{p}{\theta }}(\mathbb{R}^{n})$,\ $k\in \mathbb{Z}$, $1<\theta <p<\infty 
$ with 
\begin{equation*}
A_{\frac{p}{\theta }}(t_{k}^{p})\leq c,\quad k\in \mathbb{Z},
\end{equation*}%
where $c$ is a positive constant independent of $k$. $\newline
\mathrm{(ii)}$ The property $\mathrm{\eqref{key-est1}}$ can be generalized
in the following way. Let $1<\theta \leq p<\infty $ and $\{t_{k}\}$ be a $p$
-admissible sequence such that $t_{k}^{p}\in A_{\frac{p}{\theta }}(\mathbb{R}%
^{n})$, $k\in \mathbb{Z}$.\newline
$\mathrm{\bullet }$ If $t_{k}^{p}$,\ $k\in \mathbb{Z}$ satisfies $\mathrm{%
\eqref{Asum1}}$, then 
\begin{equation*}
\big\|\mathcal{M(}f_{j})|L_{p}(\mathbb{R}^{n},t_{k})\big\|\leq c\text{ }%
2^{\alpha _{1}(k-j)}\big\|f_{j}|L_{p}(\mathbb{R}^{n},t_{j})\big\|
\end{equation*}%
holds for all sequences of functions $f_{j}\in L_{p}(\mathbb{R}^{n},t_{j})$, 
$j\in \mathbb{Z}$ and $j\geq k$, where $c>0$ is independent of $k$ and $j$,
see \cite{D20}. $\newline
\mathrm{\bullet }$ If $t_{k}^{p}$,\ $k\in \mathbb{Z}$ satisfies $\mathrm{%
\eqref{Asum2}}$ with $\sigma _{2}\geq p$, then 
\begin{equation*}
\big\|\mathcal{M(}f_{j})|L_{p}(\mathbb{R}^{n},t_{k})\big\|\leq c\text{ }%
2^{\alpha _{2}(k-j)}\big\|f_{j}|L_{p}(\mathbb{R}^{n},t_{j})\big\|
\end{equation*}%
holds for all sequences of functions $f_{j}\in L_{p}(\mathbb{R}^{n},t_{j})$, 
$j\in \mathbb{Z}$ and $k\geq j$, where $c>0$ is independent of $k$ and $j$,
see \cite{D20}. $\newline
\mathrm{(iii)}$ A proof of Lemma \ref{key-estimate1.2} for $t_{k}^{p}=\omega 
$, $k\in \mathbb{Z}$ may be found in \cite{Mu72}.$\newline
\mathrm{(iv)}$\ In Lemma \ref{key-estimate1.2}\ we can assume that $%
t_{k}^{p}\in A_{p}(\mathbb{R}^{n})$,\ $k\in \mathbb{Z}$, $1<p<\infty $ with 
\begin{equation*}
A_{p}(t_{k}^{p})\leq c,\quad k\in \mathbb{Z},
\end{equation*}%
which follow by Lemma \ref{Ap-Property}/(iii).
\end{remark}

Let recall the vector-valued maximal inequality of Fefferman and Stein \cite%
{FeSt71}.

\begin{theorem}
\label{FS-inequality}Let $0<p<\infty ,0<q\leq \infty $ and $0<\sigma <\min
(p,q)$. Then%
\begin{equation*}
\Big\|\Big(\sum\limits_{k=-\infty }^{\infty }\big(\mathcal{M}_{\sigma
}(f_{k})\big)^{q}\Big)^{\frac{1}{q}}|L_{p}(\mathbb{R}^{n})\Big\|\lesssim %
\Big\|\Big(\sum\limits_{k=-\infty }^{\infty }\left\vert f_{k}\right\vert ^{q}%
\Big)^{\frac{1}{q}}|L_{p}(\mathbb{R}^{n})\Big\|
\end{equation*}%
holds for all sequence of functions $\{f_{k}\}\in L_{p}(\ell _{q})$.
\end{theorem}

We state one of the main tools of this paper.

\begin{lemma}
\label{key-estimate1}Let $1<\theta \leq p<\infty $\ and $1<q\leq \infty $.
Let $\{t_{k}\}$\ be a $p$-admissible\ weight\ sequence\ such that $%
t_{k}^{p}\in A_{\frac{p}{\theta }}(\mathbb{R}^{n})$, $k\in \mathbb{Z}$.
Assume that $t_{k}^{p}$,\ $k\in \mathbb{Z}$ have the same Muckenhoupt
constant, $A_{\frac{p}{\theta }}(t_{k})=c,k\in \mathbb{Z}$. Then 
\begin{equation}
\Big\|\Big(\sum\limits_{k=-\infty }^{\infty }t_{k}^{q}\big(\mathcal{M}(f_{k})%
\big)^{q}\Big)^{\frac{1}{q}}|L_{p}(\mathbb{R}^{n})\Big\|\lesssim \Big\|\Big(%
\sum\limits_{k=-\infty }^{\infty }t_{k}^{q}\left\vert f_{k}\right\vert ^{q}%
\Big)^{\frac{1}{q}}|L_{p}(\mathbb{R}^{n})\Big\|  \label{key-est}
\end{equation}%
holds for all sequences of functions $\{t_{k}f_{k}\}\in L_{p}(\ell _{q})$.
\end{lemma}

\begin{proof}
Since \eqref{key-est} when $1<q<\infty $ is given in \cite{D20.1}, we only
need consider the case $q=\infty $. For the sequence of functions $\{g_{k}\}$
we define%
\begin{equation}
T(\{g_{k}\})=\{\mathcal{M}(t_{k}\mathcal{M}(t_{k}^{-1}g_{k}))\}.
\label{operator1}
\end{equation}%
$T$ is a sublinear operator.

\textit{Step 1.} Let $K\subset \mathbb{R}^{n}$ be a compact set, $%
1<p_{2}<\infty ,\{g_{k}\}\subset L_{p_{2}}(\ell _{\infty })$\ and\ $%
t_{k}^{p_{2}}\in A_{p_{2}}(\mathbb{R}^{n})$, $k\in \mathbb{Z}$. By H\"{o}%
lder's inequality and Lemma \ref{key-estimate1.2}, we see that%
\begin{align*}
\int_{K}t_{k}(x)\mathcal{M}(t_{k}^{-1}g_{k})(x)dx\leq & \big\|\chi
_{K}|L_{p_{2}^{\prime }}(\mathbb{R}^{n})\big\|\big\|t_{k}\mathcal{M}%
(t_{k}^{-1}g_{k})|L_{p_{2}}(\mathbb{R}^{n})\big\| \\
\lesssim & \big\|g_{k}|L_{p_{2}}(\mathbb{R}^{n})\big\| \\
\lesssim & \big\|\sup_{k\in \mathbb{Z}}(|g_{k}|)|L_{p_{2}}(\mathbb{R}^{n})%
\big\| \\
\lesssim & 1,
\end{align*}%
where the implicit constant is independent of $k$. This yields $t_{k}%
\mathcal{M}(t_{k}^{-1}g_{k})\lesssim \mathcal{M}(t_{k}\mathcal{M}%
(t_{k}^{-1}g_{k}))$ almost everywhere with the implicit constant is
independent of $k$.

\textit{Step 2. }We will prove that $T$ maps $L_{p_{2}}(\ell _{\infty })$ to 
$L_{p_{2},\infty }(\ell _{\infty }),1<p_{2}<\infty $, $t_{k}^{p_{2}}\in
A_{p_{2}}(\mathbb{R}^{n})$, $k\in \mathbb{Z}$ have the same Muckenhoupt
constant, $A_{p_{2}}(t_{k}^{p_{2}})=c,k\in \mathbb{Z}$. By Lemma \ref%
{Ap-Property}/(ii)-(iii) there exists a $1<p_{3}<p_{2}<\infty $ such that $%
t_{k}^{p_{2}}\in A_{p_{3}}(\mathbb{R}^{n})$ and $t_{k}^{p_{3}}\in A_{p_{3}}(%
\mathbb{R}^{n})$. First we prove that $T$ maps $L_{p_{3}}(\ell _{\infty })$
to $L_{p_{3},\infty }(\ell _{\infty })$. Let $\{g_{k}\}\subset
L_{p_{3}}(\ell _{\infty })$ and $\delta >0$. We set 
\begin{equation*}
E_{\delta }=\{x\in \mathbb{R}^{n}:\sup_{k\in \mathbb{Z}}\mathcal{M}(t_{k}%
\mathcal{M}(t_{k}^{-1}g_{k}))(x)>\delta \}.
\end{equation*}%
$E_{\delta }$ is open set and by Step 1%
\begin{equation*}
\{x\in \mathbb{R}^{n}:\sup_{k\in \mathbb{Z}}t_{k}(x)\mathcal{M}%
(t_{k}^{-1}g_{k})(x)>c\delta \}\subset E_{\delta }
\end{equation*}%
for some positive constant $c$ independent of $\delta $ and $x$. Indeed, for 
$x\in E_{\delta }$ there is an $k_{0}\in \mathbb{Z}$ such that $\mathcal{M}%
(t_{k_{0}}\mathcal{M}(t_{k_{0}}^{-1}g_{k_{0}}))(x)>\delta $. Then there is
an open cube $Q^{x,k_{0}}$ that contains $x$ such that $%
M_{Q^{x,k_{0}}}(t_{k_{0}}\mathcal{M}(f_{k_{0}}))>\delta $. We see that%
\begin{equation*}
\sup_{k\in \mathbb{Z}}\mathcal{M}(t_{k}\mathcal{M}(t_{k}^{-1}g_{k}))(y)\geq
M_{Q^{x,k_{0}}}(t_{k_{0}}\mathcal{M}(t_{k_{0}}^{-1}g_{k_{0}}))>\delta ,\quad
y\in Q^{x,k_{0}},
\end{equation*}%
which yields that $Q^{x,k_{0}}\subset E_{\delta }$ and $E_{\delta }$ is open
set. Observe that%
\begin{equation*}
E_{\delta }\subset \bigcup_{k=-\infty }^{\infty }\Big\{x\in \mathbb{R}^{n}:%
\mathcal{M}(t_{k}\mathcal{M}(t_{k}^{-1}g_{k}))(x)>\delta \Big\}.
\end{equation*}%
Let $K_{\delta }\subset E_{\delta }$ be a compact set. We have%
\begin{equation*}
K_{\delta }\subset \bigcup_{k=0}^{N}\Big\{x\in \mathbb{R}^{n}:\mathcal{M}%
(t_{k}\mathcal{M}(t_{k}^{-1}g_{k}))(x)>\delta \Big\},\quad N\in \mathbb{N}.
\end{equation*}%
We have%
\begin{align*}
|K_{\delta }|\leq & \sum_{k=0}^{N}|\{x\in \mathbb{R}^{n}:\mathcal{M}(t_{k}%
\mathcal{M}(t_{k}^{-1}g_{k}))(x)>\delta \}| \\
=& \delta ^{-p_{3}}\sum_{k=0}^{N}\int_{\{x\in \mathbb{R}^{n}:\mathcal{M}%
(t_{k}\mathcal{M}(t_{k}^{-1}g_{k}))(x)>\delta \}}\delta ^{p_{3}}dx \\
<& \delta ^{-p_{3}}\sum_{k=0}^{N}\int_{\mathbb{R}^{n}}(\mathcal{M}(t_{k}%
\mathcal{M}(t_{k}^{-1}g_{k}))(x))^{p_{3}}dx.
\end{align*}%
By Lemma \ref{key-estimate1.1}, we see that%
\begin{align*}
\int_{\mathbb{R}^{n}}(\mathcal{M}(t_{k}\mathcal{M}%
(t_{k}^{-1}g_{k}))(x))^{p_{3}}dx=& \big\|\mathcal{M}(t_{k}\mathcal{M}%
(t_{k}^{-1}g_{k}))|L_{p_{3}}(\mathbb{R}^{n})\big\|^{p_{3}} \\
\lesssim & \big\|t_{k}\mathcal{M}(t_{k}^{-1}g_{k})|L_{p_{3}}(\mathbb{R}^{n})%
\big\|^{p_{3}} \\
\lesssim & \big\|g_{k}|L_{p_{3}}(\mathbb{R}^{n})\big\|^{p_{3}} \\
\lesssim & \big\|\sup_{k\in \mathbb{Z}}(|g_{k}|)|L_{p_{3}}(\mathbb{R}^{n})%
\big\|^{p_{3}}.
\end{align*}%
Then%
\begin{equation*}
\delta |K_{\delta }|^{\frac{1}{p_{3}}}\lesssim \big\|\sup_{k\in \mathbb{Z}%
}(|g_{k}|)|L_{p_{3}}(\mathbb{R}^{n})\big\|
\end{equation*}%
for any $\delta >0$.\ Taking the supremum over all compact $K_{\delta
}\subset $ $E_{\delta }$ and using the inner regularity of Lebesgue measure,
we obtain%
\begin{equation}
\sup_{\delta >0}\delta |E_{\delta }|^{\frac{1}{p_{3}}}\lesssim \big\|%
\sup_{k\in \mathbb{Z}}(|g_{k}|)|L_{p_{3}}(\mathbb{R}^{n})\big\|.
\label{estimate1}
\end{equation}%
Thus $T$ maps $L_{p_{3}}(\ell _{\infty })$ to $L_{p_{3},\infty }(\ell
_{\infty })$. Using again $t_{k}^{p_{2}}\in A_{p_{2}}(\mathbb{R}^{n})$, $%
k\in \mathbb{Z}$, by Theorem \ref{reverse Holder inequality} there exists an 
$\varepsilon >0$, independent on $k$, such that $t_{k}^{p_{2}(1+\varepsilon
)}\in A_{p_{2}}(\mathbb{R}^{n})\subset A_{p_{2}(1+\varepsilon )}(\mathbb{R}%
^{n}),k\in \mathbb{Z}$. Similarly to the estimate \eqref{estimate1}, we
obtain%
\begin{equation*}
\sup_{\delta >0}\delta |E_{\delta }|^{\frac{1}{p_{2}(1+\varepsilon )}%
}\lesssim \big\|\sup_{k\in \mathbb{Z}}(|g_{k}|)|L_{p_{2}(1+\varepsilon )}(%
\mathbb{R}^{n})\big\|,
\end{equation*}%
whenever $\{g_{k}\}\subset L_{p_{2}(1+\varepsilon )}(\ell _{\infty })$.
Hence $T$ maps $L_{p_{2}(1+\varepsilon )}(\ell _{\infty })$ to $%
L_{p_{2}(1+\varepsilon ),\infty }(\ell _{\infty })$. By the vector-valued
version of the Marcinkiewicz interpolation; see \cite[Lemma 1]{BCP} and \cite%
[Lemma 2.5]{GLD} we obtain the boundedness of $T$ from $L_{p_{2}}(\ell
_{\infty })$ to $L_{p_{2}}(\ell _{\infty })$ and 
\begin{align*}
\big\|\sup_{k\in \mathbb{Z}}(t_{k}\mathcal{M}(t_{k}^{-1}g_{k}))|L_{p_{2}}(%
\mathbb{R}^{n})\big\|\lesssim & \big\|\sup_{k\in \mathbb{Z}}(\mathcal{M}%
(t_{k}\mathcal{M}(t_{k}^{-1}g_{k})))|L_{p_{2}}(\mathbb{R}^{n})\big\| \\
\lesssim & \big\|\sup_{k\in \mathbb{Z}}(|g_{k}|)|L_{p_{2}}(\mathbb{R}^{n})%
\big\|,
\end{align*}%
whenever $\{g_{k}\}\subset L_{p_{2}}(\ell _{\infty })$, where the first
inequality follows by Step 1.

\textit{Step 3. }We will prove \eqref{key-est}. Let $T$ be a sublinear
operator defined in \eqref{operator1}. Let $\varepsilon _{t_{k}^{p}}>0,k\in 
\mathbb{Z}$ be as in Theorem \ref{reverse Holder inequality}. Since $%
t_{k}^{p}$,\ $k\in \mathbb{Z}$ have the same Muckenhoupt constant, $A_{\frac{%
p}{\theta }}(t_{k}^{p})=c,k\in \mathbb{Z}$, we have $\varepsilon
_{t_{k}^{p}}=\varepsilon ,k\in \mathbb{Z}$. Again, by Lemma \ref{Ap-Property}%
/(ii)-(iii) there exists a $1<p_{1}<p<\infty $ such that $t_{k}^{p}\in
A_{p_{1}}(\mathbb{R}^{n})$ and $t_{k}^{p_{1}}\in A_{p_{1}}(\mathbb{R}^{n})$.
By Theorem \ref{reverse Holder inequality} there exists an $\varepsilon >0$,
independent on $k$, such that $t_{k}^{p(1+\varepsilon )}\in A_{p}(\mathbb{R}%
^{n})\subset A_{p(1+\varepsilon )}(\mathbb{R}^{n}),k\in \mathbb{Z}$. Step 2
yields the boundedness of $T$ from $L_{p}(\ell _{\infty })$ to $L_{p}(\ell
_{\infty })$ and 
\begin{align}
\big\|\sup_{k\in \mathbb{Z}}(t_{k}\mathcal{M}(t_{k}^{-1}g_{k}))|L_{p}(%
\mathbb{R}^{n})\big\|\lesssim & \big\|\sup_{k\in \mathbb{Z}}(\mathcal{M}%
(t_{k}\mathcal{M}(t_{k}^{-1}g_{k})))|L_{p}(\mathbb{R}^{n})\big\|
\label{main-est3} \\
\lesssim & \big\|\sup_{k\in \mathbb{Z}}(|g_{k}|)|L_{p}(\mathbb{R}^{n})\big\|,
\notag
\end{align}%
whenever $\{g_{k}\}\subset L_{p}(\ell _{\infty })$, where the first
inequality follows by Step 1. In \eqref{main-est3} setting $%
g_{k}=t_{k}f_{k},k\in \mathbb{Z}$\ we then have complete the proof of %
\eqref{key-est}. This completes the proof of Lemma \ref{key-estimate1}.
\end{proof}

The next lemmas are important for the study of our function spaces.

\begin{lemma}
\label{key-estimate1.1}Let $v\in \mathbb{Z}$, $K\geq 0,1<\theta \leq
p<\infty ,1<q\leq \infty \ $ and $\alpha =(\alpha _{1},\alpha _{2})\in 
\mathbb{R}^{2}$. Let $\{t_{k}\}\in \dot{X}_{\alpha ,\sigma ,p}$ be a $p$%
-admissible weight sequence with $\sigma =(\sigma _{1}=\theta \left( \frac{p%
}{\theta }\right) ^{\prime },\sigma _{2}\geq p)$. Then for all sequence of
functions $\{t_{k}f_{k}\}\in L_{p}(\ell _{q})$, 
\begin{align}
& \Big\|\Big(\sum\limits_{k=-\infty }^{\infty }t_{k}^{q}\Big(\sum_{j=-\infty
}^{k+v}2^{(j-k)K}\mathcal{M}(f_{j})\Big)^{q}\Big)^{\frac{1}{q}}|L_{p}(%
\mathbb{R}^{n})\Big\|  \notag \\
& \lesssim \Big\|\Big(\sum\limits_{k=-\infty }^{\infty }t_{k}^{q}\left\vert
f_{k}\right\vert ^{q}\Big)^{\frac{1}{q}}|L_{p}(\mathbb{R}^{n})\Big\|
\label{key-est1.1.1}
\end{align}%
if $K>\alpha _{2}$ and 
\begin{align}
& \Big\|\Big(\sum\limits_{k=-\infty }^{\infty }t_{k}^{q}\Big(%
\sum_{j=k+v}^{\infty }2^{(j-k)K}\mathcal{M}(f_{j})\Big)^{q}\Big)^{\frac{1}{q}%
}|L_{p}(\mathbb{R}^{n})\Big\|  \notag \\
& \lesssim \Big\|\Big(\sum\limits_{k=-\infty }^{\infty }t_{k}^{q}\left\vert
f_{k}\right\vert ^{q}\Big)^{\frac{1}{q}}|L_{p}(\mathbb{R}^{n})\Big\|
\label{key-est1.1.2.}
\end{align}%
if $K<\alpha _{1}$.
\end{lemma}

\begin{proof}
As the proof for \eqref{key-est1.1.2.} is similar, we only consider %
\eqref{key-est1.1.1} with $q=\infty $, since the case that $1<q<\infty $ is
given in \cite{D20.1}.

\textit{Step 1.} For the sequence of functions $\{g_{k}\}$ we define%
\begin{equation}
T(\{g_{k}\})=\Big\{\mathcal{M}\big(t_{k}\sum_{j=-\infty }^{k+v}2^{(j-k)K}%
\mathcal{M}(t_{j}^{-1}g_{j})\big)\Big\}.  \label{operator2}
\end{equation}%
$T$ is a sublinear operator. Let $K\subset \mathbb{R}^{n}$ be a compact set, 
$1<p_{2}<\infty $, $\{g_{k}\}\subset L_{p_{2}}(\ell _{\infty })$ and $%
t_{k}^{p_{2}}\in A_{p_{2}}(\mathbb{R}^{n})$, $k\in \mathbb{Z}$. By H\"{o}%
lder's inequality and Lemma \ref{key-estimate1.2} combined with Remark \ref%
{r-estimates copy(2)}, we see that%
\begin{align}
& \int_{K}t_{k}(x)\sum_{j=-\infty }^{k+v}2^{(j-k)K}\mathcal{M}%
((t_{j}^{-1}g_{j})(x)dx  \notag \\
\leq & \big\|\chi _{K}|L_{p_{2}^{\prime }}(\mathbb{R}^{n})\big\|\Big\|%
t_{k}\sum_{j=-\infty }^{k+v}2^{(j-k)K}\mathcal{M}(t_{j}^{-1}g_{j})|L_{p_{2}}(%
\mathbb{R}^{n})\Big\|  \notag \\
\lesssim & \sum_{j=-\infty }^{k+v}2^{(j-k)K}\big\|t_{k}\mathcal{M}%
(t_{j}^{-1}g_{j})|L_{p_{2}}(\mathbb{R}^{n})\big\|  \label{term2} \\
\lesssim & \sum_{j=-\infty }^{k+v}2^{(j-k)(K-\alpha _{2})}\big\|t_{j}%
\mathcal{M}(t_{j}^{-1}g_{j})|L_{p_{2}}(\mathbb{R}^{n})\big\|  \notag \\
\lesssim & \big\|\sup_{j\in \mathbb{Z}}(|g_{j}|)|L_{p_{2}}(\mathbb{R}^{n})%
\big\|  \notag \\
\lesssim & 1,  \notag
\end{align}%
if $v\leq 0$, where the implicit constant is independent of $k$. Assume that 
$v\in \mathbb{N}$. We split the right-hand side of \eqref{term2} into two
terms, i.e.,%
\begin{equation}
\sum_{j=-\infty }^{k}2^{(j-k)K}\big\|t_{k}\mathcal{M}%
(t_{j}^{-1}g_{j})|L_{p_{2}}(\mathbb{R}^{n})\big\|%
+\sum_{j=k+1}^{k+v}2^{(j-k)K}\big\|t_{k}\mathcal{M}%
(t_{j}^{-1}g_{j})|L_{p_{2}}(\mathbb{R}^{n})\big\|.  \label{term3}
\end{equation}%
Observe that we can estimate the first term by 
\begin{equation}
c\big\|\sup_{j\in \mathbb{Z}}(|g_{j}|)|L_{p_{2}}(\mathbb{R}^{n})\big\|.
\label{term4}
\end{equation}%
Again by Lemma \ref{key-estimate1.2} combined with Remark \ref{r-estimates
copy(2)} we find that the second term of \eqref{term3}\ can be estimate by
the term in \eqref{term4}. This yields 
\begin{equation*}
t_{k}\sum_{j=-\infty }^{k+v}2^{(j-k)K}\mathcal{M}(t_{j}^{-1}g_{j})\lesssim 
\mathcal{M}\big(t_{k}\sum_{j=-\infty }^{k+v}2^{(j-k)K}\mathcal{M}%
(t_{j}^{-1}g_{j})\big)
\end{equation*}%
almost everywhere with the implicit constant is independent of $k$.

\textit{Step 2. }We will prove \eqref{key-est1.1.1}. Let $T$ be a sublinear
operator defined in \eqref{operator2}. We claim that $T$ maps $%
L_{p_{1}}(\ell _{\infty })$ to $L_{p_{1},\infty }(\ell _{\infty })$ and $%
L_{p(1+\varepsilon )}(\ell _{\infty })$ to $L_{p(1+\varepsilon ),\infty
}(\ell _{\infty })$ for some $1<p_{1}<p$. The vector-valued version of the
Marcinkiewicz interpolation; see \cite[Lemma 1]{BCP} and \cite[Lemma 2.5]%
{GLD} yields the boundedness of $T$ from $L_{p}(\ell _{\infty })$ to $%
L_{p}(\ell _{\infty })$. Moreover, whenever $\{g_{k}\}\subset L_{p}(\ell
_{\infty })$ and $t_{k}^{p}\in A_{p}(\mathbb{R}^{n})$, $k\in \mathbb{Z}$,%
\begin{align}
& \Big\|\sup_{k\in \mathbb{Z}}\Big(t_{k}\sum_{j=-\infty }^{k+v}2^{(j-k)K}%
\mathcal{M}(t_{j}^{-1}g_{j})\Big)|L_{p}(\mathbb{R}^{n})\Big\|
\label{main-est4} \\
\lesssim & \Big\|\sup_{k\in \mathbb{Z}}\Big(\mathcal{M}\big(%
t_{k}\sum_{j=-\infty }^{k+v}2^{(j-k)K}\mathcal{M}(t_{j}^{-1}g_{j})\big)\Big)%
|L_{p}(\mathbb{R}^{n})\Big\|  \notag \\
\lesssim & \big\|\sup_{j\in \mathbb{Z}}(|g_{j}|)|L_{p}(\mathbb{R}^{n})\big\|,
\notag
\end{align}%
where the first inequality follows by Step 1. In \eqref{main-est4} setting $%
g_{j}=t_{j}f_{j},j\in \mathbb{Z}$\ we then have complete the proof of %
\eqref{key-est1.1.1}.

\textit{Step 3. }We prove the claim. First we prove that $T$ maps $%
L_{p_{1}}(\ell _{\infty })$ to $L_{p_{1},\infty }(\ell _{\infty })$. Since $%
t_{k}^{p}\in A_{\frac{p}{\theta }}(\mathbb{R}^{n})\subset A_{p}(\mathbb{R}%
^{n})$, $k\in \mathbb{Z}$, by Lemma \ref{Ap-Property}/(ii)-(iii) there
exists a $1<p_{1}<p<\infty $ such that $t_{k}^{p}\in A_{p_{1}}(\mathbb{R}%
^{n})$ and $t_{k}^{p_{1}}\in A_{p_{1}}(\mathbb{R}^{n})$. Let $%
\{g_{k}\}\subset L_{p_{1}}(\ell _{\infty })$ and $\delta >0$. We set 
\begin{equation*}
E_{\delta }=\Big\{x\in \mathbb{R}^{n}:\sup_{k\in \mathbb{Z}}\Big(\mathcal{M}%
\big(t_{k}\sum_{j=-\infty }^{k+v}2^{(j-k)K}\mathcal{M}(t_{j}^{-1}g_{j})\big)%
(x)\Big)>\delta \Big\}.
\end{equation*}%
By Step 1, 
\begin{equation*}
\Big\{x\in \mathbb{R}^{n}:\sup_{k\in \mathbb{Z}}\Big(t_{k}(x)\sum_{j=-\infty
}^{k+v}2^{(j-k)K}\mathcal{M}(t_{j}^{-1}g_{j})\big)(x)\Big)>c\delta \Big\}%
\subset E_{\delta }
\end{equation*}%
for some positive constant $c$ independent of $\delta $ and $x$. $E_{\delta
} $ is open set. Indeed, for $x\in E_{\delta }$ there is an $k_{0}\in 
\mathbb{Z}$ such that 
\begin{equation*}
\mathcal{M}\big(t_{k_{0}}\sum_{j=-\infty }^{k_{0}+v}2^{(j-k)K}\mathcal{M}%
(t_{j}^{-1}g_{j})\big)(x)>\delta .
\end{equation*}%
Then there is an open cube $Q^{x,k_{0}}$ that contains $x$ such that 
\begin{equation*}
M_{Q^{x,k_{0}}}(t_{k_{0}}\sum_{j=-\infty }^{k_{0}+v}2^{(j-k)K}\mathcal{M}%
(t_{j}^{-1}g_{j}))>\delta .
\end{equation*}%
Therefore%
\begin{equation*}
\mathcal{M}\big(t_{k_{0}}\sum_{j=-\infty }^{k_{0}+v}2^{(j-k)K}\mathcal{M}%
(t_{j}^{-1}g_{j})\big)(y)\geq M_{Q_{x,k_{0}}}\big(t_{k_{0}}\sum_{j=-\infty
}^{k_{0}+v}2^{(j-k)K}\mathcal{M}(t_{j}^{-1}g_{j})\big)>\delta
\end{equation*}%
for any $y\in Q^{x,k_{0}}$. Then $Q^{x,k_{0}}\subset E_{\delta }$. Observe
that%
\begin{equation*}
E_{\delta }\subset \bigcup_{k=-\infty }^{\infty }\Big\{x\in \mathbb{R}^{n}:%
\mathcal{M}\big(t_{k}\sum_{j=-\infty }^{k+v}2^{(j-k)K}\mathcal{M}%
(t_{j}^{-1}g_{j})\big)>\delta \Big\}.
\end{equation*}%
Let $K_{\delta }\subset E_{\delta }$ be a compact set. Then%
\begin{equation*}
K_{\delta }\subset \bigcup_{k=0}^{N}\Big\{x\in \mathbb{R}^{n}:\mathcal{M}%
\big(t_{k}\sum_{j=-\infty }^{k+v}2^{(j-k)K}\mathcal{M}(t_{j}^{-1}g_{j})\big)%
>\delta \Big\},\quad N\in \mathbb{N}.
\end{equation*}%
We have%
\begin{align*}
|K_{\delta }|\leq & \sum_{k=0}^{N}\Big|\Big\{x\in \mathbb{R}^{n}:\mathcal{M}%
\big(t_{k}\sum_{j=-\infty }^{k+v}2^{(j-k)K}\mathcal{M}(t_{j}^{-1}g_{j})\big)%
(x)>\delta \Big\}\Big| \\
=& \delta ^{-p_{1}}\sum_{k=0}^{N}\int_{\big\{x\in \mathbb{R}^{n}:\mathcal{M}%
\big(t_{k}\sum_{j=-\infty }^{k+v}2^{(j-k)K}\mathcal{M}(t_{j}^{-1}g_{j})\big)%
>\delta \big\}}\delta ^{p_{1}}dx \\
<& \delta ^{-p_{1}}\sum_{k=0}^{N}\int_{\mathbb{R}^{n}}\Big(\mathcal{M}\big(%
t_{k}\sum_{j=-\infty }^{k+v}2^{(j-k)K}\mathcal{M}(t_{j}^{-1}g_{j})\big)(x)%
\Big)^{p_{1}}dx.
\end{align*}%
By Lemma \ref{key-estimate1.1}, we see that%
\begin{align*}
\int_{\mathbb{R}^{n}}\Big(\mathcal{M}\big(t_{k}\sum_{j=-\infty
}^{k+v}2^{(j-k)K}\mathcal{M}(t_{j}^{-1}g_{j})\big)\Big)^{p_{1}}dx=& \Big\|%
\mathcal{M}\big(t_{k}\sum_{j=-\infty }^{k+v}2^{(j-k)K}\mathcal{M}%
(t_{j}^{-1}g_{j})\big)|L_{p_{1}}(\mathbb{R}^{n})\Big\|^{p_{1}} \\
\lesssim & \Big(\sum_{j=-\infty }^{k+v}2^{(j-k)K}\big\|t_{k}\mathcal{M}%
(t_{j}^{-1}g_{j})|L_{p_{1}}(\mathbb{R}^{n})\big\|\Big)^{p_{1}} \\
\lesssim & \big\|\sup_{j\in \mathbb{Z}}(g_{j})|L_{p_{1}}(\mathbb{R}^{n})%
\big\|^{p_{1}}.
\end{align*}%
Then%
\begin{equation*}
\delta |K_{\delta }|^{\frac{1}{p_{1}}}\lesssim \big\|\sup_{j\in \mathbb{Z}%
}(|g_{j}|)|L_{p_{1}}(\mathbb{R}^{n})\big\|
\end{equation*}%
for any $\delta >0$. Taking the supremum over all compact $K_{\delta
}\subset $ $E_{\delta }$ and using the inner regularity of Lebesgue measure,
we obtain%
\begin{equation}
\sup_{\delta >0}\delta |E_{\delta }|^{\frac{1}{p_{1}}}\lesssim \big\|%
\sup_{j\in \mathbb{Z}}(|g_{j}|)|L_{p_{1}}(\mathbb{R}^{n})\big\|.
\label{est5}
\end{equation}%
Using again $t_{k}^{p}\in A_{\frac{p}{\theta }}(\mathbb{R}^{n})\subset A_{p}(%
\mathbb{R}^{n})$, $k\in \mathbb{Z}$, by Theorem \ref{reverse Holder
inequality} there exists an $\varepsilon >0$, independent on $k$, such that $%
t_{k}^{p(1+\varepsilon )}\in A_{p}(\mathbb{R}^{n})\subset A_{p(1+\varepsilon
)}(\mathbb{R}^{n}),k\in \mathbb{Z}$. Let $\{g_{k}\}\subset
L_{p(1+\varepsilon )}(\ell _{\infty })$. Similarly to the estimate $\mathrm{%
\eqref{est5}}$, we obtain%
\begin{equation*}
\sup_{\delta >0}\delta |E_{\delta }|^{\frac{1}{p(1+\varepsilon )}}\lesssim %
\big\|\sup_{j\in \mathbb{Z}}(|g_{j}|)|L_{p(1+\varepsilon )}(\mathbb{R}^{n})%
\big\|.
\end{equation*}%
This proves the claim. The proof is complete.
\end{proof}

\begin{remark}
\label{r-estimates copy(1)}Let $i\in \mathbb{Z},1<\theta \leq p<\infty
,1<q\leq \infty \ $and $\alpha =(\alpha _{1},\alpha _{2})\in \mathbb{R}^{2}$%
. Let $\{t_{k}\}\in \dot{X}_{\alpha ,\sigma ,p}$ be a $p$-admissible weight
sequence with $\sigma =(\sigma _{1}=\theta \left( \frac{p}{\theta }\right)
^{\prime },\sigma _{2}\geq p)$. From Lemma \ref{key-estimate1.1} we easily
obtain 
\begin{equation*}
\Big\|\Big(\sum\limits_{k=-\infty }^{\infty }t_{k}^{q}\big(\mathcal{M}%
(f_{k+i})\big)^{q}\Big)^{\frac{1}{q}}|L_{p}(\mathbb{R}^{n})\Big\|\lesssim %
\Big\|\Big(\sum\limits_{k=-\infty }^{\infty }t_{k}^{q}\left\vert
f_{k}\right\vert ^{q}\Big)^{\frac{1}{q}}|L_{p}(\mathbb{R}^{n})\Big\|
\end{equation*}%
holds for all sequence of functions $\{t_{k}f_{k}\}\in L_{p}(\ell _{q})$,
where the implicit constant depends on $i$. Indeed, we have 
\begin{equation*}
\mathcal{M}(f_{k+i})\leq \sum\limits_{j=-\infty }^{k+i}2^{(j-k-i)M}\mathcal{M%
}(f_{j}),\quad M>\alpha _{2},k\in \mathbb{Z}.
\end{equation*}%
Lemma \ref{key-estimate1.1} yields the desired result.
\end{remark}

\begin{remark}
\label{r-estimates}$\mathrm{(i)}$ We would like to mention that the result
of this section is true if we assume that $t_{k}\in A_{\frac{p}{\theta }}(%
\mathbb{R}^{n})$,\ $k\in \mathbb{Z}$, $1<p<\infty $ with%
\begin{equation*}
A_{\frac{p}{\theta }}(t_{k})\leq c,\quad k\in \mathbb{Z},
\end{equation*}%
where $c>0$ independent of $k$.$\newline
\mathrm{(ii)}$ A proof of this result\ for $t_{k}^{p}=\omega $, $k\in 
\mathbb{Z}$ may be found in \cite{AJ80} and \cite{Kok78}.$\newline
\mathrm{(iii)}$ In view of Lemma \ref{Ap-Property}/(i) we can assume that $%
t_{k}^{p}\in A_{p}(\mathbb{R}^{n})$,\ $k\in \mathbb{Z}$, $1<p<\infty $ with%
\begin{equation*}
A_{p}(t_{k}^{p})\leq c,\quad k\in \mathbb{Z},
\end{equation*}%
where $c>0$ independent of $k$.
\end{remark}

\section{The spaces ${\dot{A}_{p,q}(\mathbb{R}^{n},\{t_{k}\})}$}

Select a pair of Schwartz functions $\varphi $ and $\psi $ satisfy%
\begin{equation}
\text{supp}\mathcal{F}\varphi ,\mathcal{F}\psi \subset \big\{\xi :\frac{1}{2}%
\leq |\xi |\leq 2\big\},  \label{Ass1}
\end{equation}%
\begin{equation}
|\mathcal{F}\varphi (\xi )|,|\mathcal{F}\psi (\xi )|\geq c\quad \text{if}%
\quad \frac{3}{5}\leq |\xi |\leq \frac{5}{3}  \label{Ass2}
\end{equation}%
and 
\begin{equation}
\sum_{k=-\infty }^{\infty }\overline{\mathcal{F}\varphi (2^{-k}\xi )}%
\mathcal{F}\psi (2^{-k}\xi )=1\quad \text{if}\quad \xi \neq 0,  \label{Ass3}
\end{equation}%
where $c>0$. Throughout the paper, for all $k$ $\in \mathbb{Z}$ and $x\in 
\mathbb{R}^{n}$, we put $\varphi _{k}(x):=2^{kn}\varphi (2^{k}x)$. Let $%
\varphi \in \mathcal{S}(\mathbb{R}^{n})$ be a function satisfying $\mathrm{%
\eqref{Ass1}}$-$\mathrm{\eqref{Ass2}}$. Recall that there exists a function $%
\psi \in \mathcal{S}(\mathbb{R}^{n})$ satisfying $\mathrm{\eqref{Ass1}}$-$%
\mathrm{\eqref{Ass3}}$, see \cite[Lemma (6.9)]{FrJaWe01}. We set%
\begin{equation*}
\mathcal{S}_{\infty }(\mathbb{R}^{n}):=\Big\{\varphi \in \mathcal{S}(\mathbb{%
R}^{n}):\int_{\mathbb{R}^{n}}x^{\beta }\varphi (x)dx=0\text{ for all
multi-indices }\beta \in \mathbb{N}_{0}^{n}\Big\}.
\end{equation*}%
Let $\mathcal{S}_{\infty }^{\prime }(\mathbb{R}^{n})$ be the topological
dual of $\mathcal{S}_{\infty }(\mathbb{R}^{n})$, namely, the set of all
continuous linear functionals on $\mathcal{S}_{\infty }(\mathbb{R}^{n})$.

Now, we define the spaces under consideration.

\begin{definition}
\label{B-F-def}Let $0<p\leq \infty $ and $0<q\leq \infty $. Let $\{t_{k}\}$
be a $p$-admissible weight sequence, and $\varphi \in \mathcal{S}(\mathbb{R}%
^{n})$\ satisfy $\mathrm{\eqref{Ass1}}$ and $\mathrm{\eqref{Ass2}}$.\newline
$\mathrm{(i)}$ The Besov space $\dot{B}_{p,q}(\mathbb{R}^{n},\{t_{k}\})$\ is
the collection of all $f\in \mathcal{S}_{\infty }^{\prime }(\mathbb{R}^{n})$%
\ such that 
\begin{equation*}
\big\|f|\dot{B}_{p,q}(\mathbb{R}^{n},\{t_{k}\})\big\|:=\Big(%
\sum\limits_{k=-\infty }^{\infty }\big\|t_{k}(\varphi _{k}\ast f)|L_{p}(%
\mathbb{R}^{n})\big\|^{q}\Big)^{\frac{1}{q}}<\infty
\end{equation*}%
with the usual modifications if $q=\infty $.\newline
$\mathrm{(ii)}$ Let $0<p<\infty $. The Triebel-Lizorkin space $\dot{F}_{p,q}(%
\mathbb{R}^{n},\{t_{k}\})$\ is the collection of all $f\in \mathcal{S}%
_{\infty }^{\prime }(\mathbb{R}^{n})$\ such that 
\begin{equation*}
\big\|f|\dot{F}_{p,q}(\mathbb{R}^{n},\{t_{k}\})\big\|:=\Big\|\Big(%
\sum\limits_{k=-\infty }^{\infty }t_{k}^{q}|\varphi _{k}\ast f|^{q}\Big)^{%
\frac{1}{q}}|L_{p}(\mathbb{R}^{n})\Big\|<\infty
\end{equation*}%
with the usual modifications if $q=\infty $.
\end{definition}

Now, we define the spaces $\dot{F}_{\infty ,q}(\mathbb{R}^{n},\{t_{k}\})$,
see \cite{D20.3}.

\begin{definition}
Let $0<q<\infty $.\ Let $\{t_{k}\}$ be a $q$-admissible weight sequence, and 
$\varphi \in \mathcal{S}(\mathbb{R}^{n})$\ satisfy $\mathrm{\eqref{Ass1}}$
and $\mathrm{\eqref{Ass2}}$. The Triebel-Lizorkin\ space $\dot{F}_{\infty
,q}(\mathbb{R}^{n},\{t_{k}\})$\ is the collection of all $f\in \mathcal{S}%
_{\infty }^{\prime }(\mathbb{R}^{n})$\ such that 
\begin{equation*}
{\big\|}f|\dot{F}_{\infty ,q}(\mathbb{R}^{n},\{t_{k}\}){\big\|}:=\sup_{P\in 
\mathcal{Q}}\Big(\frac{1}{|P|}\int_{P}\sum\limits_{k=-\log _{2}l(P)}^{\infty
}t_{k}^{q}(x)|\varphi _{k}\ast f(x)|^{q}dx\Big)^{\frac{1}{q}}<\infty .
\end{equation*}
\end{definition}

\begin{remark}
Some properties of these function spaces, such as the $\varphi $-transform
characterization in the sense of Frazier and Jawerth, duality, the smooth
atomic and molecular decomposition and the characterization of Besov spaces $%
\dot{B}_{p,q}(\mathbb{R}^{n},\{t_{k}\})$\ in terms of the difference
relations are given in\ \cite{D20}, \cite{D20.1} and \cite{D20.3}.
\end{remark}

\begin{remark}
We would like to mention that the elements of the above spaces are not
distributions but equivalence classes of distributions$.$ We will use $\dot{A%
}_{p,q}(\mathbb{R}^{n},\{t_{k}\})$ to denote either $\dot{B}_{p,q}(\mathbb{R}%
^{n},\{t_{k}\})$ or $\dot{F}_{p,q}(\mathbb{R}^{n},\{t_{k}\})$.
\end{remark}

Using the system $\{\varphi _{k}\}_{k\in \mathbb{Z}}$ we can define the
quasi-norms%
\begin{equation*}
\big\|f|\dot{B}_{p,q}^{s}(\mathbb{R}^{n})\big\|:=\Big(\sum\limits_{k=-\infty
}^{\infty }2^{ksq}\big\|\varphi _{k}\ast f|L_{p}(\mathbb{R}^{n})\big\|^{q}%
\Big)^{\frac{1}{q}}
\end{equation*}%
and%
\begin{equation*}
\big\|f|\dot{F}_{p,q}^{s}(\mathbb{R}^{n})\big\|:=\Big\|\Big(%
\sum\limits_{k=-\infty }^{\infty }2^{ksq}|\varphi _{k}\ast f|^{q}\Big)^{%
\frac{1}{q}}|L_{p}(\mathbb{R}^{n})\Big\|
\end{equation*}%
for constants $s\in \mathbb{R}$ and $0<p,q\leq \infty $ with $0<p<\infty $
in the $\dot{F}$-case. The Besov space $\dot{B}_{p,q}^{s}(\mathbb{R}^{n})$\
consist of all distributions $f\in \mathcal{S}_{\infty }^{\prime }(\mathbb{R}%
^{n})$ for which $\big\|f|\dot{B}_{p,q}^{s}(\mathbb{R}^{n})\big\|<\infty .$
The Triebel-Lizorkin space $\dot{F}_{p,q}^{s}(\mathbb{R}^{n})$\ consist of
all distributions $f\in \mathcal{S}_{\infty }^{\prime }(\mathbb{R}^{n})$ for
which $\big\|f|\dot{F}_{p,q}^{s}(\mathbb{R}^{n})\big\|<\infty .$ Further
details on the classical theory of these spaces can be found in \cite{FJ86}, 
\cite{FJ90} and \cite{T1}.

One recognizes immediately that if $\{t_{k}\}=\{2^{sk}\}$, $s\in \mathbb{R}$%
, then 
\begin{equation*}
\dot{B}_{p,q}(\mathbb{R}^{n},\{2^{sk}\})=\dot{B}_{p,q}^{s}(\mathbb{R}%
^{n}),\quad \dot{F}_{p,q}(\mathbb{R}^{n},\{2^{sk}\})=\dot{F}_{p,q}^{s}(%
\mathbb{R}^{n})
\end{equation*}%
and%
\begin{equation*}
\dot{F}_{\infty ,q}(\mathbb{R}^{n},\{2^{sk}\})=\dot{F}_{\infty ,q}^{s}(%
\mathbb{R}^{n}).
\end{equation*}%
Moreover, for $\{t_{k}\}=\{2^{sk}w\}$, $s\in \mathbb{R}$ with a weight $w$
we re-obtain the weighted Triebel-Lizorkin spaces; we refer to the papers 
\cite{Bui82} and \cite{Tang} for a comprehensive treatment of the weighted
spaces.

Let $\varphi $, $\psi \in \mathcal{S}(\mathbb{R}^{n})$ satisfying %
\eqref{Ass1}\ through\ \eqref{Ass3}. Recall that the $\varphi $-transform $%
S_{\varphi }$ is defined by setting $(S_{\varphi }f)_{k,m}=\langle f,\varphi
_{k,m}\rangle $ where $\varphi _{k,m}(x)=2^{\frac{kn}{2}}\varphi (2^{k}x-m)$%
, $m\in \mathbb{Z}^{n}$ and $k\in \mathbb{Z}$. The inverse $\varphi $%
-transform $T_{\psi }$ is defined by 
\begin{equation*}
T_{\psi }\lambda :=\sum_{k=-\infty }^{\infty }\sum_{m\in \mathbb{Z}%
^{n}}\lambda _{k,m}\psi _{k,m},
\end{equation*}%
where $\lambda =\{\lambda _{k,m}\}_{k\in \mathbb{Z},m\in \mathbb{Z}%
^{n}}\subset \mathbb{C}$, see \cite{FJ90}.

Now we introduce the following sequence spaces.

\begin{definition}
\label{sequence-space}Let $0<p\leq \infty $ and $0<q\leq \infty $. Let $%
\{t_{k}\}$ be a $p$-admissible weight sequence. Then for all complex valued
sequences $\lambda =\{\lambda _{k,m}\}_{k\in \mathbb{Z},m\in \mathbb{Z}%
^{n}}\subset \mathbb{C}$ we define%
\begin{equation*}
\dot{b}_{p,q}(\mathbb{R}^{n},\{t_{k}\}):=\Big\{\lambda :\big\|\lambda |\dot{b%
}_{p,q}(\mathbb{R}^{n},\{t_{k}\})\big\|<\infty \Big\},
\end{equation*}%
where%
\begin{equation*}
\big\|\lambda |\dot{b}_{p,q}(\mathbb{R}^{n},\{t_{k}\})\big\|:=\Big(%
\sum_{k=-\infty }^{\infty }2^{\frac{knq}{2}}\big\|\sum\limits_{m\in \mathbb{Z%
}^{n}}t_{k}\lambda _{k,m}\chi _{k,m}|L_{p}(\mathbb{R}^{n})\big\|^{q}\Big)^{%
\frac{1}{q}}
\end{equation*}%
and 
\begin{equation*}
\dot{f}_{p,q}(\mathbb{R}^{n},\{t_{k}\}):=\Big\{\lambda :\big\|\lambda |\dot{f%
}_{p,q}(\mathbb{R}^{n},\{t_{k}\})\big\|<\infty \Big\}
\end{equation*}%
with $0<p<\infty $, where%
\begin{equation*}
\big\|\lambda |\dot{f}_{p,q}(\mathbb{R}^{n},\{t_{k}\})\big\|:=\Big\|\Big(%
\sum_{k=-\infty }^{\infty }\sum\limits_{m\in \mathbb{Z}^{n}}2^{\frac{knq}{2}%
}t_{k}^{q}|\lambda _{k,m}|^{q}\chi _{k,m}\Big)^{\frac{1}{q}}|L_{p}(\mathbb{R}%
^{n})\Big\|.
\end{equation*}
\end{definition}

Allowing the smoothness $t_{k}$, $k\in \mathbb{Z}$ to vary from point to
point will raise extra difficulties\ to study these function spaces. But by
the following lemma the problem can be reduced to the case of fixed
smoothness, see {\cite{D20.2}}.

\begin{proposition}
\label{Equi-norm1}Let $0<p\leq \infty $ and $0<q\leq \infty $. Let $%
\{t_{k}\} $ be a $p$-admissible weight sequence.\newline
$\mathrm{(i)}$ Then%
\begin{equation*}
\big\|\lambda |\dot{b}_{p,q}(\mathbb{R}^{n},\{t_{k}\})\big\|^{\ast }:=\Big(%
\sum_{k=-\infty }^{\infty }2^{\frac{knq}{2}}\Big(\sum\limits_{m\in \mathbb{Z}%
^{n}}|\lambda _{k,m}|^{p}t_{k,m}^{p}\Big)^{\frac{q}{p}}\Big)^{\frac{1}{q}},
\end{equation*}%
\textrm{\ }is an equivalent quasi-norm in $\dot{b}_{p,q}(\mathbb{R}%
^{n},\{t_{k}\})$.\newline
$\mathrm{(ii)}$ Let $0<\theta \leq p<\infty $, $0<q<\infty $. Assume that $%
\{t_{k}\}$ satisfying $\mathrm{\eqref{Asum1}}$ with $\sigma _{1}=\theta
\left( \frac{p}{\theta }\right) ^{\prime }$ and $j=k$. Then%
\begin{equation*}
\big\|\lambda |\dot{f}_{p,q}(\mathbb{R}^{n},\{t_{k}\})\big\|^{\ast }:=\Big\|%
\Big(\sum_{k=-\infty }^{\infty }\sum\limits_{m\in \mathbb{Z}^{n}}2^{knq(%
\frac{1}{2}+\frac{1}{p})}t_{k,m}^{q}|\lambda _{k,m}|^{q}\chi _{k,m}\Big)^{%
\frac{1}{q}}|L_{p}(\mathbb{R}^{n})\Big\|,
\end{equation*}%
is an equivalent quasi-norm in $\dot{f}_{p,q}(\mathbb{R}^{n},\{t_{k}\})$,
where%
\begin{equation*}
t_{k,m}:=\big\|t_{k}|L_{p}(Q_{k,m})\big\|,\quad k\in \mathbb{Z},m\in \mathbb{%
Z}^{n}.
\end{equation*}
\end{proposition}

We define $\dot{f}_{\infty ,q}(\mathbb{R}^{n},\{t_{k}\})$, the sequence
space corresponding to\ $\dot{F}_{\infty ,q}(\mathbb{R}^{n},\{t_{k}\})$ as
follows.

\begin{definition}
Let $0<q<\infty $ and $\{t_{k}\}$ be a $q$-admissible sequence. Then for all
complex valued sequences $\lambda =\{\lambda _{k,m}\}_{k\in \mathbb{Z},m\in 
\mathbb{Z}^{n}}\subset \mathbb{C}$ we define 
\begin{equation*}
\dot{f}_{\infty ,q}(\mathbb{R}^{n},\{t_{k}\}):=\Big\{\lambda :{\big\|}%
\lambda |\dot{f}_{\infty ,q}(\mathbb{R}^{n},\{t_{k}\}){\big\|}<\infty \Big\},
\end{equation*}%
where%
\begin{equation}
{\big\|}\lambda |\dot{f}_{\infty ,q}(\mathbb{R}^{n},\{t_{k}\}){\big\|}%
:=\sup_{P\in \mathcal{Q}}\Big(\frac{1}{|P|}\int_{P}\sum\limits_{k=-\log
_{2}l(P)}^{\infty }\sum\limits_{m\in \mathbb{Z}^{n}}2^{\frac{knq}{2}%
}t_{k}^{q}(x)|\lambda _{k,m}|^{q}\chi _{k,m}(x)dx\Big)^{\frac{1}{q}}.
\label{norm}
\end{equation}
\end{definition}

The quasi-norm \eqref{norm} can be rewritten as follows:

\begin{proposition}
\label{equi-q=infinity}Let $0<q<\infty $. Let $\{t_{k}\}$ be a $q$%
-admissible sequence. Then 
\begin{equation*}
\big\|\lambda |\dot{f}_{\infty ,q}(\mathbb{R}^{n},\{t_{k}\})\big\|%
=\sup_{P\in \mathcal{Q}}\Big(\frac{1}{|P|}\int_{P}\sum\limits_{k=-\log
_{2}l(P)}^{\infty }\sum\limits_{m\in \mathbb{Z}^{n}}2^{knq(\frac{1}{2}+\frac{%
1}{q})}t_{k,m,q}^{q}|\lambda _{k,m}|^{q}\chi _{k,m}(x)dx\Big)^{\frac{1}{q}},
\end{equation*}%
where%
\begin{equation*}
t_{k,m,q}:=\big\|t_{k}|L_{q}(Q_{k,m})\big\|,\quad k\in \mathbb{Z},m\in 
\mathbb{Z}^{n}.
\end{equation*}
\end{proposition}

For simplicity, in what follows, we use $\dot{a}_{p,q}(\mathbb{R}%
^{n},\{t_{k}\})$ to denote either $\dot{b}_{p,q}(\mathbb{R}^{n},\{t_{k}\})$
or $\dot{f}_{p,q}(\mathbb{R}^{n},\{t_{k}\})$. Now we have the following
result which is called the $\varphi $-transform characterization in the
sense of Frazier and Jawerth. It will play an important role in the rest of
the paper. Based on Lemmas \ref{key-estimate1} and \ref{key-estimate1.1} the
proof is similar to that of {\cite{D20} and \cite{D20.1} }

\begin{theorem}
\label{phi-tran}Let $\alpha =(\alpha _{1},\alpha _{2})\in \mathbb{R}%
^{2},0<\theta \leq p<\infty $ and$\ 0<q\leq \infty $. Let $\{t_{k}\}\in \dot{%
X}_{\alpha ,\sigma ,p}$ be a $p$-admissible weight sequence with $\sigma
=(\sigma _{1}=\theta \left( \frac{p}{\theta }\right) ^{\prime },\sigma
_{2}\geq p)$.\ Let $\varphi $, $\psi \in \mathcal{S}(\mathbb{R}^{n})$
satisfying $\mathrm{\eqref{Ass1}}$\ through\ $\mathrm{\eqref{Ass3}}$. The
operators 
\begin{equation*}
S_{\varphi }:\dot{A}_{p,q}(\mathbb{R}^{n},\{t_{k}\})\rightarrow \dot{a}%
_{p,q}(\mathbb{R}^{n},\{t_{k}\})
\end{equation*}%
and 
\begin{equation*}
T_{\psi }:\dot{a}_{p,q}(\mathbb{R}^{n},\{t_{k}\})\rightarrow \dot{A}_{p,q}(%
\mathbb{R}^{n},\{t_{k}\})
\end{equation*}%
are bounded. Furthermore, $T_{\psi }\circ S_{\varphi }$ is the identity on $%
\dot{A}_{p,q}(\mathbb{R}^{n},\{t_{k}\})$.
\end{theorem}

\begin{corollary}
\label{Indpendent}Let $\alpha =(\alpha _{1},\alpha _{2})\in \mathbb{R}%
^{2},0<\theta \leq p<\infty $ and $0<q\leq \infty $. Let $\{t_{k}\}\in \dot{X%
}_{\alpha ,\sigma ,p}$ be a $p$-admissible weight sequence with $\sigma
=(\sigma _{1}=\theta \left( \frac{p}{\theta }\right) ^{\prime },\sigma
_{2}\geq p)$. The definition of the spaces $\dot{A}_{p,q}(\mathbb{R}%
^{n},\{t_{k}\})$ is independent of the choices of $\varphi \in \mathcal{S}(%
\mathbb{R}^{n})$ satisfying $\mathrm{\eqref{Ass1}}$\ and\ $\mathrm{%
\eqref{Ass2}}$.
\end{corollary}

\begin{theorem}
\label{Banach}Let $\alpha =(\alpha _{1},\alpha _{2})\in \mathbb{R}%
^{2},0<\theta \leq p<\infty $ and $0<q\leq \infty $. Let $\{t_{k}\}\in \dot{X%
}_{\alpha ,\sigma ,p}$ be a $p$-admissible weight sequence with $\sigma
=(\sigma _{1}=\theta \left( \frac{p}{\theta }\right) ^{\prime },\sigma
_{2}\geq p)$. $\dot{A}_{p,q}(\mathbb{R}^{n},\{t_{k}\})$ are quasi-Banach
spaces. They are Banach spaces if $1\leq p<\infty $ and $1\leq q\leq \infty $%
.
\end{theorem}

\begin{theorem}
\label{embeddings-S-inf}Let $0<\theta \leq p<\infty $ and $0<q\leq \infty $.
Let\ $\{t_{k}\}\in \dot{X}_{\alpha ,\sigma ,p}$ be a $p$-admissible weight
sequence with $\sigma =(\sigma _{1}=\theta \left( \frac{p}{\theta }\right)
^{\prime },\sigma _{2}\geq p)$ and $\alpha =(\alpha _{1},\alpha _{2})\in 
\mathbb{R}^{2}$.\ \newline
$\mathrm{(i)}$ We have the embedding%
\begin{equation*}
\mathcal{S}_{\infty }(\mathbb{R}^{n})\hookrightarrow \dot{A}_{p,q}(\mathbb{R}%
^{n},\{t_{k}\})\hookrightarrow \mathcal{S}_{\infty }^{\prime }(\mathbb{R}%
^{n}).
\end{equation*}%
In addition if $0<q<\infty $, then $\mathcal{S}_{\infty }(\mathbb{R}^{n})$
is dense in $\dot{A}_{p,q}(\mathbb{R}^{n},\{t_{k}\}\mathrm{.}$
\end{theorem}

To prove Theorems \ref{Banach} and \ref{embeddings-S-inf}, using Lemmas \ref%
{key-estimate1} and \ref{key-estimate1.1} to replace Lemmas 2.14 and 2.18 in 
{\cite{D20.1}} and repeating the proof of Theorems 3.12 and 3.13 in {\cite%
{D20.1}.}

The following statements are from {\cite{D20.3}.}

\begin{theorem}
\label{phi-tran2}Let $0<\theta \leq q<\infty $. Let $\{t_{k}\}\in \dot{X}%
_{\alpha ,\sigma ,q}$ be a $q$-admissible weight sequence with $\sigma
=(\sigma _{1}=\theta \left( \frac{q}{\theta }\right) ^{\prime },\sigma
_{2}\geq q)$. Let $\varphi $, $\psi \in \mathcal{S}(\mathbb{R}^{n})$
satisfying $\mathrm{\eqref{Ass1}}$\ through\ $\mathrm{\eqref{Ass3}}$. The
operators 
\begin{equation*}
S_{\varphi }:\dot{F}_{\infty ,q}(\mathbb{R}^{n},\{t_{k}\})\rightarrow \dot{f}%
_{\infty ,q}(\mathbb{R}^{n},\{t_{k}\})
\end{equation*}%
and 
\begin{equation*}
T_{\psi }:\dot{f}_{\infty ,q}(\mathbb{R}^{n},\{t_{k}\})\rightarrow \dot{F}%
_{\infty ,q}(\mathbb{R}^{n},\{t_{k}\})
\end{equation*}%
are bounded. Furthermore, $T_{\psi }\circ S_{\varphi }$ is the identity on $%
\dot{F}_{\infty ,q}(\mathbb{R}^{n},\{t_{k}\})$.
\end{theorem}

\begin{corollary}
Let $0<\theta \leq q<\infty $. Let $\{t_{k}\}\in \dot{X}_{\alpha ,\sigma ,q}$
be a $q$-admissible weight sequence with $\sigma =(\sigma _{1}=\theta \left( 
\frac{q}{\theta }\right) ^{\prime },\sigma _{2}\geq q)$. The definition of
the spaces $\dot{F}_{\infty ,q}(\mathbb{R}^{n},\{t_{k}\})$ is independent of
the choices of $\varphi \in \mathcal{S}(\mathbb{R}^{n})$ satisfying $\mathrm{%
\eqref{Ass1}}$\ through\ $\mathrm{\eqref{Ass2}}$.
\end{corollary}

\begin{remark}
\label{phi-tran1}Theorems \ref{phi-tran} and \ref{phi-tran2} can then be
exploited to obtain a variety of results for the $\dot{A}_{p,q}(\mathbb{R}%
^{n},\{t_{k}\})$ spaces, where arguments can be equivalently transferred to
the sequence space, which is often more convenient to handle. More
precisely, under the same hypothesis of Theorems \ref{phi-tran} and \ref%
{phi-tran2}, we obtain 
\begin{equation*}
\big\|\{\langle f,\varphi _{k,m}\rangle \}_{k\in \mathbb{Z},m\in \mathbb{Z}%
^{n}}|\dot{a}_{p,q}(\mathbb{R}^{n},\{t_{k}\})\big\|\approx \big\|f|\dot{A}%
_{p,q}(\mathbb{R}^{n},\{t_{k}\})\big\|.
\end{equation*}
\end{remark}

Let $0<\theta <p<\infty $. Let $\{t_{k}\}\in \dot{X}_{\alpha ,\sigma ,p}$ be
a $p$-admissible weight sequence with $\sigma =(\sigma _{1}=\theta \left( 
\frac{p}{\theta }\right) ^{\prime },\sigma _{2}\geq p)$ and $\alpha =(\alpha
_{1},\alpha _{2})\in \mathbb{R}^{2}$. For simplicity, in what follows, when $%
\alpha =(\alpha _{1},\alpha _{2})=(0,0)$ we write $\{t_{k}\}\in \dot{X}%
_{\sigma ,p}.$

Now, we state one of the main result of this section.

\begin{theorem}
\label{new-norm}Let $0<\theta <p<\infty ,0<q\leq \infty $ and $j\in \mathbb{Z%
}$. Let\ $\{t_{k}\}\in \dot{X}_{\sigma ,p}$ be a $p$-admissible weight
sequence with $\sigma =(\sigma _{1}=\theta \left( \frac{p}{\theta }\right)
^{\prime },\sigma _{2}\geq p)$. Then 
\begin{equation*}
\dot{A}_{p,q}(\mathbb{R}^{n},\{t_{k}\})=\dot{A}_{p,q}(\mathbb{R}^{n},t_{j})
\end{equation*}%
in the sense of equivalent quasi-norms.
\end{theorem}

\begin{proof}
As the proof for $\dot{B}_{p,q}(\mathbb{R}^{n},\{t_{k}\})$ is similar, we
only condider $\dot{F}_{p,q}(\mathbb{R}^{n},\{t_{k}\})$. We will do the
proof into two steps.

\textit{Step 1.} In this step we prove that 
\begin{equation}
\big\|f|\dot{F}_{p,q}(\mathbb{R}^{n},\{t_{k}\})\big\|\lesssim \big\|f|\dot{F}%
_{p,q}(\mathbb{R}^{n},t_{j})\big\|  \label{step1}
\end{equation}%
for any $f\in \dot{F}_{p,q}(\mathbb{R}^{n},t_{j})$, where the implicit
constant is independent of $j$. We have%
\begin{equation*}
\big\|f|\dot{F}_{p,q}(\mathbb{R}^{n},\{t_{k}\})\big\|\approx \big\|\{\langle
f,\varphi _{k,m}\rangle \}_{k\in \mathbb{Z},m\in \mathbb{Z}^{n}}|\dot{f}%
_{p,q}(\mathbb{R}^{n},\{t_{k}\})\big\|,
\end{equation*}%
see Remark \ref{phi-tran1}. In view of the definition of $\varphi _{k,m}$,
for every $k\in \mathbb{Z}$ and every $m\in \mathbb{Z}^{n}$ we have%
\begin{equation*}
|\langle f,\varphi _{k,m}\rangle |=2^{-\frac{kn}{2}}|f\ast \tilde{\varphi}%
_{k}(2^{-k}m)|\leq 2^{-\frac{kn}{2}}\sup_{z\in Q_{k,m}}|f\ast \tilde{\varphi}%
_{k}(z)|,
\end{equation*}%
where $\tilde{\varphi}=\varphi (-\cdot )$. We recall the following estimate
see (2.11) in \cite{FJ86}, 
\begin{equation}
\sup_{z\in Q_{k,m}}\left\vert \tilde{\varphi}_{k}\ast f(z)\right\vert
\lesssim 2^{k\frac{n}{\tau }}\Big(\sum_{h\in \mathbb{Z}^{n}}(1+\left\vert
h\right\vert )^{-M}\int_{Q_{k,m+h}}\left\vert \tilde{\varphi}_{k}\ast
f(y)\right\vert ^{\tau }dy\Big)^{\frac{1}{\tau }},  \label{phi2}
\end{equation}%
with $M>n+1,\tau >0$. If we choose $\frac{1}{\tau }=\frac{1}{\nu }+\frac{1}{%
\sigma _{1}}$, with $0<\nu <\min (p,q)$, then by H\"{o}lder's inequality we
obtain that 
\begin{equation}
2^{k\frac{n}{\tau }}\Big(\int_{Q_{k,m+h}}\left\vert \tilde{\varphi}_{k}\ast
f(y)\right\vert ^{\tau }dy\Big)^{\frac{1}{\tau }}\leq M_{Q_{k,m+h},\nu
}(t_{j}(\tilde{\varphi}_{k}\ast f))M_{Q_{k,m+h},\sigma _{1}}(t_{j}^{-1}).
\label{phi3}
\end{equation}%
From \eqref{Asum1},%
\begin{equation*}
M_{Q_{k,u},\sigma _{1}}(t_{j}^{-1})\leq C\big(M_{Q_{k,u},p}(t_{k})\big)^{-1},
\end{equation*}%
if $k\leq j$ and $u\in \mathbb{Z}^{n}$, where the positive constant $C$ is
independent of $j,k$ and $u$. Using the fact that%
\begin{equation*}
Q_{k,m}\subset B(x_{k,m},\sqrt{n}2^{-k})\quad \text{and}\quad
Q_{k,m+h}\subset B(x_{k,m},\sqrt{n}(1+\left\vert h\right\vert )2^{-k}),
\end{equation*}%
we conclude that%
\begin{align}
M_{Q_{k,m},p}(t_{k})M_{Q_{k,m+h},\sigma _{1}}(t_{j}^{-1})\leq &
(1+\left\vert h\right\vert )^{\frac{n}{\theta }%
}M_{Q_{k,m+h},p}(t_{k})M_{Q_{k,m+h},\sigma _{1}}(t_{j}^{-1})  \notag \\
\leq & c(1+\left\vert h\right\vert )^{\frac{n}{\theta }},  \label{phi1.1}
\end{align}%
where $c>0$ is independent of $j,k,h$ and $m$. Substituting \eqref{phi3},
with the help of \eqref{phi1.1}, into \eqref{phi2}, we obtain 
\begin{align*}
(M_{Q_{k,m},p}(t_{k}))^{\tau }(\sup_{z\in Q_{k,m}}\left\vert \tilde{\varphi}%
_{k}\ast f(z)\right\vert )^{\tau }\lesssim & \sum_{h\in \mathbb{Z}%
^{n}}(1+\left\vert h\right\vert )^{\frac{n\tau }{\theta }-M}\big(%
M_{Q_{k,m+h},\nu }(t_{j}(\tilde{\varphi}_{k}\ast f))\big)^{\tau } \\
\lesssim & (\mathcal{M}_{\nu }(t_{j}(\tilde{\varphi}_{k}\ast f))(x))^{\tau
}\sum_{h\in \mathbb{Z}^{n}}(1+\left\vert h\right\vert )^{\frac{n\tau }{%
\theta }+n\tau -M} \\
\lesssim & (\mathcal{M}_{\nu }(t_{j}(\tilde{\varphi}_{k}\ast f))(x))^{\tau }
\end{align*}%
for any $x\in Q_{k,m},k\leq j,m\in \mathbb{Z}^{n}$ and any $M$ large enough.
Now if $k>j$, then by \eqref{Asum2}, we derive that $M_{Q_{k,m},p}(t_{k})%
\lesssim M_{Q_{k,m},p}(t_{j})$. Therefore%
\begin{equation*}
M_{Q_{k,m},p}(t_{k})\sup_{z\in Q_{k,m}}\left\vert \tilde{\varphi}_{k}\ast
f(z)\right\vert \lesssim \mathcal{M}_{\nu }(t_{j}(\tilde{\varphi}_{k}\ast
f))(x)
\end{equation*}%
for any $x\in Q_{k,m},k\in \mathbb{Z},m\in \mathbb{Z}^{n}$. Observe that we
can apply Proposition \ref{Equi-norm1} and Theorem \ref{FS-inequality}.
Consequently%
\begin{align*}
& \big\|\{\langle f,\varphi _{k,m}\rangle \}_{k\in \mathbb{Z},m\in \mathbb{Z}%
^{n}}|\dot{f}_{p,q}(\mathbb{R}^{n},\{t_{k}\})\big\| \\
\approx & \Big\|\Big(\sum_{k=-\infty }^{\infty }\sum_{m\in \mathbb{Z}%
^{n}}2^{kn(\frac{1}{2}+\frac{1}{p})q}t_{k,m}^{q}|\langle f,\varphi
_{k,m}\rangle |^{q}\chi _{k,m}\Big)^{\frac{1}{q}}|L_{p}(\mathbb{R}^{n})\Big\|
\\
\lesssim & \Big\|\Big(\sum\limits_{k=-\infty }^{\infty }\big(\mathcal{M}%
_{\nu }(t_{j}(\tilde{\varphi}_{k}\ast f))\big)^{q}\Big)^{\frac{1}{q}}|L_{p}(%
\mathbb{R}^{n})\Big\| \\
\lesssim & \Big\|\Big(\sum\limits_{k=-\infty }^{\infty }|\tilde{\varphi}%
_{k}\ast f|^{q}\Big)^{\frac{1}{q}}t_{j}|L_{p}(\mathbb{R}^{n})\Big\| \\
\lesssim & \big\|f|\dot{F}_{p,q}(\mathbb{R}^{n},t_{j})\big\|,
\end{align*}%
where the implicit constant is independent of $j$.

\textit{Step 2.} In this step we prove the opposite inequality of %
\eqref{step1}. We have%
\begin{equation*}
\big\|f|\dot{F}_{p,q}(\mathbb{R}^{n},t_{j})\big\|\approx \big\|\{\langle
f,\varphi _{k,m}\rangle \}_{k\in \mathbb{Z},m\in \mathbb{Z}^{n}}|\dot{f}%
_{p,q}(\mathbb{R}^{n},t_{j})\big\|.
\end{equation*}%
We keep the estimate \eqref{phi2}. We choose again $\frac{1}{\tau }=\frac{1}{%
\nu }+\frac{1}{\sigma _{1}}$, with $0<\nu <\min (p,q)$. By H\"{o}lder's
inequality we obtain that 
\begin{equation*}
2^{k\frac{n}{\tau }}\Big(\int_{Q_{k,m+h}}\left\vert \tilde{\varphi}_{k}\ast
f(y)\right\vert ^{\tau }dy\Big)^{\frac{1}{\tau }}\leq M_{Q_{k,m+h},\nu
}(t_{k}(\tilde{\varphi}_{k}\ast f))M_{Q_{k,m+h},\sigma _{1}}(t_{k}^{-1}).
\end{equation*}%
From \eqref{Asum2}, $M_{Q_{k,m},p}(t_{j})\leq CM_{Q_{k,m},p}(t_{k})$ if $%
k\leq j$ and $m\in \mathbb{Z}^{n}$, where the positive constant $C$ is
independent of $j,k$ and $m$. Observe that%
\begin{equation*}
M_{Q_{k,m},p}(t_{k})M_{Q_{k,m+h},\sigma _{1}}(t_{k}^{-1})\lesssim
(1+\left\vert h\right\vert )^{\frac{n}{\theta }},
\end{equation*}%
where the implicit constant is independent of $j,k,h$ and $m$. We can argue
as in Step 1 and obtain%
\begin{equation}
M_{Q_{k,m},p}(t_{j})\sup_{z\in Q_{k,m}}\left\vert \tilde{\varphi}_{k}\ast
f(z)\right\vert \lesssim \mathcal{M}_{\nu }(t_{k}(\tilde{\varphi}_{k}\ast
f))(x)  \label{phi5}
\end{equation}%
for any $x\in Q_{k,m},k\leq j,m\in \mathbb{Z}^{n}$. Now if $k>j$, then by %
\eqref{Asum1}, we obtain that $M_{Q_{k,m+h},\sigma _{1}}(t_{k}^{-1})\leq C%
\big(M_{Q_{k,m+h},p}(t_{j})\big)^{-1}$. This yields the same estimate %
\eqref{phi5} for any $k>j$. Again we can apply Proposition \ref{Equi-norm1}
and Theorem\ \ref{FS-inequality} and find 
\begin{align*}
& \big\|\{\langle f,\varphi _{k,m}\rangle \}_{k\in \mathbb{Z},m\in \mathbb{Z}%
^{n}}|\dot{f}_{p,q}(\mathbb{R}^{n},t_{j})\big\| \\
\approx & \Big\|\Big(\sum_{k=-\infty }^{\infty }\sum_{m\in \mathbb{Z}%
^{n}}2^{k\frac{nq}{2}}(M_{Q_{k,m},p}(t_{j}))^{q}|\langle f,\varphi
_{k,m}\rangle |^{q}\chi _{k,m}\Big)^{\frac{1}{q}}|L_{p}(\mathbb{R}^{n})\Big\|
\\
\lesssim & \Big\|\Big(\sum\limits_{k=-\infty }^{\infty }\big(\mathcal{M}%
_{\nu }(t_{k}(\tilde{\varphi}_{k}\ast f))\big)^{q}\Big)^{\frac{1}{q}}|L_{p}(%
\mathbb{R}^{n})\Big\| \\
\lesssim & \Big\|\Big(\sum\limits_{k=-\infty }^{\infty }t_{k}^{q}|\tilde{%
\varphi}_{k}\ast f|^{q}\Big)^{\frac{1}{q}}|L_{p}(\mathbb{R}^{n})\Big\| \\
\lesssim & \big\|f|\dot{F}_{p,q}(\mathbb{R}^{n},\{t_{k}\})\big\|,
\end{align*}%
where the implicit constant is independent of $j$. This finishes the proof.
\end{proof}

Let $w$ be a locally integrable function, $w(x)>0$ for almost every $x$. Let 
$0<p<\infty $ and $H_{p}(\mathbb{R}^{n},w^{p})$ be the weighted Hardy
spaces, see \cite{ST89} for more details. Based on the relation between
weighted Hardy and Triebel-Lizorkin spaces, see \cite[Theorem 1.4]{Bui82},
Theorem \ref{new-norm}\ yields the following.

\begin{corollary}
\label{new-norm1}Let $0<\theta <p<\infty $\ and $j\in \mathbb{Z}$. Let\ $%
\{t_{k}\}\in \dot{X}_{\sigma ,p}$ be a $p$-admissible weight sequence with $%
\sigma =(\sigma _{1}=\theta \left( \frac{p}{\theta }\right) ^{\prime
},\sigma _{2}\geq p)$. Then 
\begin{equation}
\dot{F}_{p,2}(\mathbb{R}^{n},\{t_{k}\})=H_{p}(\mathbb{R}^{n},t_{j}^{p})
\label{triebel-hardy}
\end{equation}%
in the sense of equivalent quasi-norms. In particular%
\begin{equation}
\dot{F}_{p,2}(\mathbb{R}^{n},\{t_{k}\})=L_{p}(\mathbb{R}^{n},t_{j}),\quad
\max (1,\theta )<p<\infty  \label{triebel-lebesgue}
\end{equation}%
in the sense of equivalent norms.
\end{corollary}

Now we have the following necessary and sufficient conditions for the
coincidence of the weighted Lebesgue spaces $L_{p}(\mathbb{R}%
^{n},t_{i}),i\in \{1,2\}$.

\begin{theorem}
\label{Coincidence1}Let $\max (1,\theta )<p<\infty $ and $\sigma _{1}=\theta
(\frac{p}{\theta })^{\prime }$. Let\ $\{t_{1}^{p},t_{2}^{p}\}\in A_{\frac{p}{%
\theta }}(\mathbb{R}^{n})$. Then%
\begin{equation}
L_{p}(\mathbb{R}^{n},t_{1})=L_{p}(\mathbb{R}^{n},t_{2})  \label{Coincidence}
\end{equation}%
if and only if%
\begin{equation}
M_{Q,\sigma _{1}}(t_{1}^{-1})\approx M_{Q,\sigma _{1}}(t_{2}^{-1})\quad 
\text{and}\quad M_{Q,p}(t_{1})\approx M_{Q,p}(t_{2})  \label{Necessary}
\end{equation}%
for any dyadic cube $Q$ of $\mathbb{R}^{n}$. In particular, 
\begin{equation*}
L_{p}(\mathbb{R}^{n})=L_{p}(\mathbb{R}^{n},t_{1}),
\end{equation*}%
if and only if%
\begin{equation*}
M_{Q,\sigma _{1}}(t_{1}^{-1})\approx M_{Q,q}(t_{1})\approx 1
\end{equation*}%
for any dyadic cube $Q$ of $\mathbb{R}^{n}$.
\end{theorem}

\begin{proof}
We proceed in two steps. Let $Q$ be a dyadic cube of $\mathbb{R}^{n}$.

\textit{Step 1.} In this step we prove that \eqref{Coincidence} holds if and
only if%
\begin{equation}
M_{Q,p}(t_{i})M_{Q,\sigma _{1}}(t_{j}^{-1})\leq C,\quad i,j\in \{1,2\}\quad 
\text{and}\quad M_{Q,p}(t_{1})\approx M_{Q,p}(t_{2}).  \label{Necessary1}
\end{equation}%
The sufficiency part follows by Corollary \ref{new-norm1}. Suppose that %
\eqref{Coincidence} is valid. By the function $f=\chi _{Q}$, the necessity
of the second condition of \eqref{Necessary1} is now evident. Let $g$ be a
locally integrable function on $\mathbb{R}^{n}$. As in \cite[7.1.1]{L.
Graf14}, we have%
\begin{equation*}
\big(M_{Q}(g^{\theta })\big)^{\frac{1}{\theta }}\leq \big(\mathcal{M}%
(g^{\theta })(x)\big)^{\frac{1}{\theta }},\quad x\in Q.
\end{equation*}%
Thus,%
\begin{align*}
\big(M_{Q}(g^{\theta })\big)^{\frac{1}{\theta }}\big\|1|L_{p}(Q,t_{1})\big\|%
\leq & \big\|\big(\mathcal{M}(g^{\theta })\big)^{\frac{1}{\theta }}|L_{p}(%
\mathbb{R}^{n},t_{1})\big\| \\
\lesssim & \big\|\big(\mathcal{M}(g^{\theta })\big)^{\frac{1}{\theta }%
}|L_{p}(\mathbb{R}^{n},t_{2})\big\| \\
=& c\big\|\mathcal{M}(g^{\theta })|L_{\frac{p}{\theta }}(\mathbb{R}%
^{n},t_{2}^{\theta })\big\|^{\frac{1}{\theta }} \\
\lesssim & \big\|g^{\theta }|L_{\frac{p}{\theta }}(\mathbb{R}%
^{n},t_{2}^{\theta })\big\|^{\frac{1}{\theta }} \\
=& c\big\|g|L_{p}(\mathbb{R}^{n},t_{2})\big\|,
\end{align*}%
where the four inequality follows by the fact that $t_{2}^{p}\in A_{\frac{p}{%
\theta }}(\mathbb{R}^{n})$.\ We choose $g=t_{2}^{-(\frac{p}{\theta }%
)^{\prime }}\chi _{Q}$. We deduce that%
\begin{equation*}
|Q|^{\frac{1}{p}}M_{Q,p}(t_{1})\big(M_{Q}(t_{2}^{-\sigma })\big)^{\frac{1}{%
\theta }}\Big(\int_{Q}t_{2}^{p-(\frac{p}{\theta })^{\prime }p}\Big)^{-\frac{1%
}{p}}\lesssim c.
\end{equation*}%
Since $p-(\frac{p}{\theta })^{\prime }p=-\sigma _{1}$ and $\frac{1}{\sigma }+%
\frac{1}{p}=\frac{1}{\theta }$, we get the necessity of $M_{Q,p}(t_{1})M_{Q,%
\sigma _{1}}(t_{2}^{-1})\leq C$. Interchanging the roles of $t_{2}$ and $%
t_{1}$, we obtain the necessity of the first condition of \eqref{Necessary1}.

\textit{Step 2.} We prove that \eqref{Coincidence} holds if and only if%
\begin{equation*}
M_{Q,\sigma _{1}}(t_{2}^{-1})\approx M_{Q,\sigma _{1}}(t_{1}^{-1}).
\end{equation*}%
Using \eqref{triebel-lebesgue}, the sufficiency part follows by a simple
modification of the proof of Theorem \ref{new-norm}. By H\"{o}lder's
inequality we obtain that 
\begin{equation}
1=\Big(\frac{1}{Q}\int_{Q}t_{1}^{\theta }(x)t_{1}^{-\theta }(x)\Big)^{\frac{1%
}{\theta }}\leq M_{Q,p}(t_{1})M_{Q,\sigma _{1}}(t_{1}^{-1}).  \label{Holder}
\end{equation}%
Multiplying both sides of \eqref{Holder}\ by $M_{Q,\sigma _{1}}(t_{2}^{-1})$
and using Step 1, we get%
\begin{equation*}
M_{Q,\sigma _{1}}(t_{2}^{-1})\lesssim M_{Q,\sigma _{1}}(t_{1}^{-1}).
\end{equation*}%
Interchanging the roles of $t_{2}$ and $t_{1}$, we obtain the necessity of
the first condition of \eqref{Necessary}. Theorem \ref{Coincidence1} is
proved.
\end{proof}

Now we have the following necessary and sufficient conditions for the
coincidence of the weighted function spaces $\dot{A}_{p,q}(\mathbb{R}%
^{n},t_{i}),i\in \{1,2\}$.

\begin{theorem}
\label{CoincidenceBesov-Triebel}Let $0<\theta <p<\infty $\ and\ $0<q\leq
\infty $. Let\ $\{t_{1}^{p},t_{2}^{p}\}\in A_{\frac{p}{\theta }}(\mathbb{R}%
^{n})$. Then%
\begin{equation}
\dot{A}_{p,q}(\mathbb{R}^{n},t_{1})=\dot{A}_{p,q}(\mathbb{R}^{n},t_{2})
\label{Coincidence2}
\end{equation}%
if and only if \eqref{Necessary} holds. In particular, 
\begin{equation*}
\dot{A}_{p,q}(\mathbb{R}^{n})=\dot{A}_{p,q}(\mathbb{R}^{n},t_{1}),
\end{equation*}%
if and only if%
\begin{equation*}
M_{Q,\sigma _{1}}(t_{1}^{-1})\approx M_{Q,q}(t_{1})\approx 1
\end{equation*}%
for any dyadic cube $Q$ of $\mathbb{R}^{n}$.
\end{theorem}

\begin{proof}
The sufficiency part follows by Theorem \ref{new-norm}. Assume that %
\eqref{Coincidence2} holds. Let $k,k_{0}\in \mathbb{Z}$ and\ $m,m_{0}\in 
\mathbb{Z}^{n}$. Define 
\begin{equation*}
\lambda _{k,m}=\left\{ 
\begin{array}{ccc}
1, & \text{if} & (k,m)=(k_{0},m_{0}), \\ 
0, & \text{otherwise} & 
\end{array}%
\right.
\end{equation*}%
and $\lambda =\{\lambda _{k,m}\}_{k\in \mathbb{Z},m\in \mathbb{Z}^{n}}$. Let 
$\varphi $, $\psi \in \mathcal{S}(\mathbb{R}^{n})$ satisfying \eqref{Ass1}\
through\ \eqref{Ass3}. We put%
\begin{equation*}
T_{\psi }\lambda =\sum_{k\in \mathbb{Z}}\sum_{m\in \mathbb{Z}^{n}}\lambda
_{k,m}2^{k\frac{n}{2}}\psi (2^{k}x-m).
\end{equation*}%
Let $Q_{k_{0},m_{0}}$\ be a dyadic cube of $\mathbb{R}^{n}$. From Remark \ref%
{phi-tran1} , we get 
\begin{align*}
\big\|1|L_{p}(Q_{k_{0},m_{0}},t_{1})\big\|^{p}&
=\int_{Q_{k_{0},m_{0}}}t_{1}^{p}(x)dx \\
& =\frac{2^{k_{0}n}}{|\langle \psi ,\varphi \rangle |^{p}}%
\int_{Q_{k_{0},m_{0}}}|\langle \psi (2^{k_{0}}\cdot -m_{0}),\varphi
(2^{k_{0}}\cdot -m_{0})\rangle |^{p}t_{1}^{p}(x)dx \\
& \leq \frac{1}{|\langle \psi ,\varphi \rangle |^{p}}\big\|\{\langle T_{\psi
}\lambda ,\varphi _{j,h}\rangle \}_{j\in \mathbb{Z},h\in \mathbb{Z}^{n}}|%
\dot{a}_{p,q}(\mathbb{R}^{n},\{t_{1}\})\big\|^{p} \\
& \leq \frac{1}{|\langle \psi ,\varphi \rangle |^{p}}\big\|T_{\psi }\lambda |%
\dot{A}_{p,q}(\mathbb{R}^{n},\{t_{1}\})\big\|^{p} \\
& \lesssim \big\|T_{\psi }\lambda |\dot{A}_{p,q}(\mathbb{R}^{n},\{t_{2}\})%
\big\|^{p}.
\end{align*}%
Applying Theorem \ref{phi-tran1}, we obtain that 
\begin{align*}
\big\|T_{\psi }\lambda |\dot{A}_{p,q}(\mathbb{R}^{n},\{t_{2}\})\big\|&
\lesssim \big\|\{\lambda _{k,m}\}_{k\in \mathbb{Z},m\in \mathbb{Z}^{n}}|\dot{%
a}_{p,q}(\mathbb{R}^{n},\{t_{2}\})\big\| \\
& \lesssim \big\|1|L_{p}(Q_{k_{0},m_{0}},t_{2})\big\|.
\end{align*}%
Interchanging the roles of $t_{2}$ and $t_{1}$, we deduce%
\begin{equation*}
M_{Q_{k_{0},m_{0}},p}(t_{1})\approx M_{Q_{k_{0},m_{0}},p}(t_{2}),\quad
k_{0}\in \mathbb{Z},\quad m_{0}\in \mathbb{Z}.
\end{equation*}%
Now we establish the necessity of the first condition of \eqref{Necessary1}.
Similar as in Step 1 of the proof of Theorem \ref{Coincidence1}, we derive 
\begin{align*}
\big(M_{Q_{k_{0},m_{0}}}(g^{\theta })\big)^{\frac{1}{\theta }}\big\|%
1|L_{p}(Q_{k_{0},m_{0}},t_{1})\big\|& \leq \big(M_{Q_{k_{0},m_{0}}}(g^{%
\theta })\big)^{\frac{1}{\theta }}\big\|1|L_{p}(Q_{k_{0},m_{0}},t_{2})\big\|%
^{p} \\
& \lesssim \big\|\mathcal{M}(g^{\theta })|L_{\frac{p}{\theta }}(\mathbb{R}%
^{n},t_{2}^{\theta })\big\|^{\frac{1}{\theta }} \\
& \lesssim \big\|g^{\theta }|L_{\frac{p}{\theta }}(\mathbb{R}%
^{n},t_{2}^{\theta })\big\|^{\frac{1}{\theta }} \\
& =c\big\|g|L_{p}(\mathbb{R}^{n},t_{2})\big\|.
\end{align*}%
We choose $g=t_{2}^{-(\frac{p}{\theta })^{\prime }}\chi _{Q_{k_{0},m_{0}}}$.
We deduce that%
\begin{equation*}
|Q_{k_{0},m_{0}}|^{\frac{1}{p}}M_{Q_{k_{0},m_{0}},p}(t_{1})\big(%
M_{Q_{k_{0},m_{0}}}(t_{2}^{-\sigma })\big)^{\frac{1}{\theta }}\Big(%
\int_{Q_{k_{0},m_{0}}}t_{2}^{p-(\frac{p}{\theta })^{\prime }p}\Big)^{-\frac{1%
}{p}}\lesssim c,
\end{equation*}%
which yields the necessity of $%
M_{Q_{k_{0},m_{0}},p}(t_{1})M_{Q_{k_{0},m_{0}},\sigma }(t_{2}^{-1})\leq C$.
The rest is the same as in Step 2 of Theorem \ref{Coincidence1}.
\end{proof}

\begin{remark}
Clearly, we can deduce Theorem \ref{Coincidence1} by Theorem \ref%
{CoincidenceBesov-Triebel}. But to prove the necessity of \eqref{Necessary1}
for the coincidence of the weighted Lebesgue spaces $L_{p}(\mathbb{R}%
^{n},t_{i}),i\in \{1,2\}$\ we don't need to use the $\varphi $-transform
characterization of such spaces.
\end{remark}

From \eqref{triebel-hardy}\ and Theorem \ref{CoincidenceBesov-Triebel}, we
deduce the following conclusion.

\begin{corollary}
\label{CoincidenceHardy}Let $0<\theta <p<\infty $ and $\{t_{1}^{p},t_{2}^{p}%
\}\in A_{\frac{p}{\theta }}(\mathbb{R}^{n})$. Then%
\begin{equation*}
H_{p}(\mathbb{R}^{n},t_{1}^{p})=H_{p}(\mathbb{R}^{n},t_{2}^{p})
\end{equation*}%
if and only if \eqref{Necessary} holds. In particular, 
\begin{equation*}
H_{p}(\mathbb{R}^{n})=H_{p}(\mathbb{R}^{n},t_{1}^{p}),
\end{equation*}%
if and only if%
\begin{equation*}
M_{Q,\sigma _{1}}(t_{1}^{-1})\approx M_{Q,q}(t_{1})\approx 1
\end{equation*}%
for any dyadic cube $Q$ of $\mathbb{R}^{n}$.
\end{corollary}

\begin{remark}
The question arises what about Hardy spaces with general weights. We present
the following possibility definition to such spaces. Let $\psi \in \mathcal{S%
}(\mathbb{R}^{n})$. We set $\psi _{k}=2^{kn}\psi (2^{k}\cdot ),k\in \mathbb{Z%
},$ 
\begin{equation*}
p_{N}(\varphi ):=\sup_{\beta \in \mathbb{N}_{0}^{n},|\beta |\leq
N}\sup_{x\in \mathbb{R}^{n}}|\partial ^{\beta }\varphi (x)|(1+|x|)^{N}<\infty
\end{equation*}%
for all $N\in \mathbb{N}$ and%
\begin{equation*}
\mathcal{F}_{N}:=\{\varphi \in \mathcal{S}(\mathbb{R}^{n}):p_{N}(\varphi
)\leq 1\}.
\end{equation*}%
Let $\{t_{k}\}$ be a $p$-admissible sequence i.e. $t_{k}\in L_{p}^{\mathrm{%
loc}}(\mathbb{R}^{n}),0<p<\infty ,k\in \mathbb{Z}$ and $f\in \mathcal{S}%
^{\prime }(\mathbb{R}^{n})$. Denote by $M$ the grand maximal operator given
by%
\begin{equation*}
M(f;\{t_{k}\}):=\sup_{k\in \mathbb{Z}}\{t_{k}|\psi _{k}\ast f|:\psi \in 
\mathcal{F}_{N}\},\quad N\in \mathbb{N}.
\end{equation*}%
Let $0<p<\infty $. Then we define\ the Hardy space $\tilde{H}_{p}(\mathbb{R}%
^{n},\{t_{k}\})$\ to be the set of all $f\in \mathcal{S}^{\prime }(\mathbb{R}%
^{n})$\ such that%
\begin{equation*}
\big\|f|\tilde{H}_{p}(\mathbb{R}^{n},\{t_{k}\})\big\|:=\big\|%
M(f;\{t_{k}\})|L_{p}(\mathbb{R}^{n})\big\|<\infty .
\end{equation*}
\end{remark}

Now, we state the last result of this section.

\begin{theorem}
\label{new-norm2}Let $0<\theta <q<\infty $ and $j\in \mathbb{Z}$. Let\ $%
\{t_{k}\}\in \dot{X}_{\sigma ,q}$ be a $q$-admissible weight sequence with $%
\sigma =(\sigma _{1}=\theta \left( \frac{q}{\theta }\right) ^{\prime
},\sigma _{2}\geq q)$. Then 
\begin{equation*}
\dot{F}_{\infty ,q}(\mathbb{R}^{n},\{t_{k}\})=\dot{F}_{\infty ,q}(\mathbb{R}%
^{n},t_{j})
\end{equation*}%
in the sense of equivalent quasi-norms.
\end{theorem}

\begin{proof}
Proceed in two steps.

\textit{Step 1.} We will prove the estimate: 
\begin{equation}
\big\|f|\dot{F}_{\infty ,q}(\mathbb{R}^{n},\{t_{k}\})\big\|\lesssim \big\|f|%
\dot{F}_{\infty ,q}(\mathbb{R}^{n},t_{j})\big\|  \label{q=infinity}
\end{equation}%
for any $f\in \dot{F}_{\infty ,q}(\mathbb{R}^{n},t_{j})$, where the implicit
constant is independent of $j$. We have%
\begin{equation*}
\big\|f|\dot{F}_{\infty ,q}(\mathbb{R}^{n},\{t_{k}\})\big\|\approx \big\|%
\{\langle f,\varphi _{k,m}\rangle \}_{k\in \mathbb{Z},m\in \mathbb{Z}^{n}}|%
\dot{f}_{\infty ,q}(\mathbb{R}^{n},\{t_{k}\})\big\|.
\end{equation*}%
As in Theorem \ref{new-norm} for every $k\in \mathbb{Z}$ and every $m\in 
\mathbb{Z}^{n}$ we have%
\begin{equation*}
|\langle f,\varphi _{k,m}\rangle |\leq 2^{-\frac{kn}{2}}\sup_{z\in
Q_{k,m}}|f\ast \tilde{\varphi}_{k}(z)|,
\end{equation*}%
where $\tilde{\varphi}=\varphi (-\cdot )$. We keep the estimate \eqref{phi2}%
. We choose $\frac{1}{\tau }=\frac{1}{\nu }+\frac{1}{\sigma _{1}}$, with $%
0<\nu <q$. By H\"{o}lder's inequality we obtain that 
\begin{equation}
2^{k\frac{n}{\tau }}\Big(\int_{Q_{k,m+h}}\left\vert \tilde{\varphi}_{k}\ast
f(y)\right\vert ^{\tau }dy\Big)^{\frac{1}{\tau }}\leq M_{Q_{k,m+h},\nu
}(t_{j}(\tilde{\varphi}_{k}\ast f))M_{Q_{k,m+h},\sigma _{1}}(t_{j}^{-1}).
\label{q=infinity1}
\end{equation}%
{Let $x\in Q_{k,m}\subset P\in $}$\mathcal{Q}${\ and }$y\in Q_{k,m+h}${.
Then }%
\begin{align*}
\left\vert y-x_{P}\right\vert \leq & \left\vert y-x\right\vert +\left\vert
x-x_{P}\right\vert \\
\leq & (2\sqrt{n}+|h|)2^{-k}+\sqrt{n}2^{-k_{P}} \\
\leq & (3\sqrt{n}+|h|)2^{-k_{P}},\quad k_{P}=-\log _{2}l(P).
\end{align*}%
Let $i\in \mathbb{N}$ be such that $2^{i-1}\leq 3\sqrt{n}+|h|<2^{i}$. We
have that $y$ is located in the ball $B(x_{P},2^{i-k_{P}})$. Therefore the
right-hand side of \eqref{q=infinity1} does not exceed%
\begin{equation*}
M_{Q_{k,m+h},\nu }(t_{j}(\tilde{\varphi}_{k}\ast f)\chi
_{B(x_{P},2^{i-k_{P}})})M_{Q_{k,m+h},\sigma _{1}}(t_{j}^{-1}).
\end{equation*}%
We set $g(h)=(1+\left\vert h\right\vert )^{\frac{n\tau }{\theta }+n\tau -M}$
and $\omega (h)=(1+\left\vert h\right\vert )^{\frac{n\tau }{\theta }+n\tau
-M-\frac{\tau n}{q}},h\in \mathbb{R}^{n}$. Repeating the argument in the
proof of Theorem \ref{new-norm} we obtain that%
\begin{align*}
& (M_{Q_{k,m},q}(t_{k}))^{\tau }(\sup_{z\in Q_{k,m}}\left\vert \tilde{\varphi%
}_{k}\ast f(z)\right\vert )^{\tau } \\
\lesssim & \sum_{h\in \mathbb{Z}^{n}}(1+\left\vert h\right\vert )^{\frac{%
n\tau }{\theta }-M}\big(M_{Q_{k,m+h},\nu }(t_{j}(\tilde{\varphi}_{k}\ast
f)\chi _{B(x_{P},2^{i-k_{P}})})\big)^{\tau } \\
\lesssim & \sum_{h\in \mathbb{Z}^{n}}g(h)(\big(\mathcal{M}_{\nu }(t_{j}(%
\tilde{\varphi}_{k}\ast f)\chi _{B(x_{P},2^{i-k_{P}})})(x)\big)^{\tau }
\end{align*}%
for any $x\in Q_{k,m}{\subset P\in }\mathcal{Q},k\geq -\log _{2}l(P),m\in 
\mathbb{Z}^{n}$. Consequently%
\begin{equation*}
2^{k\frac{nq}{2}}\big((M_{Q_{k,m},q}(t_{k}))^{\tau }|\langle f,\varphi
_{k,m}\rangle |^{\tau }\big)^{\frac{q}{\tau }}\lesssim \Big(\sum_{h\in 
\mathbb{Z}^{n}}g(h)\big(\mathcal{M}_{\nu }(t_{j}(\tilde{\varphi}_{k}\ast
f)\chi _{B(x_{P},2^{i-k_{P}})})(x)\big)^{\tau }\Big)^{\frac{q}{\tau }}
\end{equation*}%
for any $x\in Q_{k,m}{\subset P\in }\mathcal{Q},k\geq -\log _{2}l(P),m\in 
\mathbb{Z}^{n}$ and%
\begin{align*}
& \Big(\sum\limits_{k=-\log _{2}l(P)}^{\infty }\sum_{m\in \mathbb{Z}^{n}}2^{k%
\frac{nq}{2}}\big((M_{Q_{k,m},q}(t_{k}))^{\tau }|\langle f,\varphi
_{k,m}\rangle |^{\tau }\big)^{\frac{q}{\tau }}\chi _{Q_{k,m}\cap P}\Big)^{%
\frac{\tau }{q}} \\
& \lesssim \sum_{h\in \mathbb{Z}^{n}}g(h)\Big(\sum\limits_{k=-\log
_{2}l(P)}^{\infty }\big(\mathcal{M}_{\nu }(t_{j}(\tilde{\varphi}_{k}\ast
f)\chi _{B(x_{P},2^{i-k_{P}})})\big)^{q}\Big)^{\frac{\tau }{q}}.
\end{align*}%
Hence, by Proposition \ref{equi-q=infinity} and Theorem\ \ref{FS-inequality}%
, we obtain%
\begin{align*}
& \big\|\{\langle f,\varphi _{k,m}\rangle \}_{k\in \mathbb{Z},m\in \mathbb{Z}%
^{n}}|\dot{f}_{\infty ,q}(\mathbb{R}^{n},\{t_{k}\})\big\|^{\tau } \\
& =\sup_{P\in \mathcal{Q}}\frac{1}{|P|^{\frac{\tau }{q}}}\Big\|\Big(%
\sum\limits_{k=-\log _{2}l(P)}^{\infty }\sum_{m\in \mathbb{Z}^{n}}2^{kn\frac{%
q}{2}}\big(M_{Q_{k,m},q}(t_{k}))^{\tau }|\langle f,\varphi _{k,m}\rangle
|^{\tau }\big)^{\frac{q}{\tau }}\chi _{Q_{k,m}\cap P}\Big)^{\frac{\tau }{q}%
}|L_{\frac{q}{\tau }}(\mathbb{R}^{n})\Big\| \\
& \lesssim \sum_{h\in \mathbb{Z}^{n}}g(h)\sup_{P\in \mathcal{Q}}\frac{1}{%
|P|^{\frac{\tau }{q}}}\Big\|\Big(\sum\limits_{k=k_{P}-i}^{\infty }\big(%
\mathcal{M}_{\nu }(t_{j}(\tilde{\varphi}_{k}\ast f)\chi
_{B(x_{P},2^{i-k_{P}})})\big)^{q}\Big)^{\frac{\tau }{q}}|L_{\frac{q}{\tau }}(%
\mathbb{R}^{n})\Big\| \\
& \lesssim \sum_{h\in \mathbb{Z}^{n}}\omega (h)\sup_{P\in \mathcal{Q}}\frac{1%
}{|B(x_{P},2^{i-k_{P}})|^{\frac{\tau }{q}}}\Big\|\Big(\sum%
\limits_{k=k_{P}-i}^{\infty }\big(t_{j}(\tilde{\varphi}_{k}\ast f)\chi
_{B(x_{P},2^{i-k_{P}})}\big)^{q}\Big)^{\frac{1}{q}}|L_{q}(\mathbb{R}^{n})%
\Big\|^{\tau } \\
& \lesssim \big\|f|\dot{F}_{\infty ,q}(\mathbb{R}^{n},t_{j})\big\|^{\tau },
\end{align*}%
by taking $M$ large enough, where the implicit constant is independent of $j$%
.

\textit{Step 2.} As in Step 1, the opposite inequality of \eqref{q=infinity}
follow by the arguments of Step 2 of Theorem \ref{new-norm}. This finishes
the proof.
\end{proof}

Let $BMO$ be the space of functions of bounded mean oscillation. From
Theorem \ref{new-norm2} and by the same arguments as used in proof of
Theorem\ \ref{CoincidenceBesov-Triebel}, we deduce the following conclusion.

\begin{corollary}
\label{BMO}Let $0<\theta <q<\infty ,\sigma _{1}=\theta (\frac{q}{\theta }%
)^{\prime }$ and $\{t_{1}^{q},t_{2}^{q}\}\in A_{\frac{q}{\theta }}(\mathbb{R}%
^{n})$. Then%
\begin{equation*}
\dot{F}_{\infty ,q}(\mathbb{R}^{n},t_{1})=\dot{F}_{\infty ,q}(\mathbb{R}%
^{n},t_{2})
\end{equation*}%
if and only if \eqref{Necessary} holds. In particular, if $t^{q}\in A_{\frac{%
q}{\theta }}(\mathbb{R}^{n})$, then 
\begin{equation*}
BMO=\dot{F}_{\infty ,q}(\mathbb{R}^{n},t),
\end{equation*}%
holds if and only if%
\begin{equation*}
M_{Q,\sigma _{1}}(t^{-1})\approx M_{Q,q}(t)\approx 1
\end{equation*}%
for any dyadic cube $Q$ of $\mathbb{R}^{n}$.
\end{corollary}

\subsection*{Acknowledgements}

This work was supported by the General Direction of Higher Education and
Training under Grant No. C00L03UN280120220004 and by The General Directorate
of Scientific Research and Technological Development, Algeria.

\end{document}